\documentclass[preprint]{elsarticle}

\usepackage{hyperref}

\usepackage{amsmath}
\usepackage{amsfonts}
\usepackage{amssymb}
\usepackage{graphicx}
\usepackage{comment}
\usepackage{subcaption}
\usepackage[ruled]{algorithm2e}
\usepackage{algorithmic}

\usepackage{exscale}

\usepackage{color}

\newtheorem{theorem}{Theorem}

\newtheorem{proposition}{Proposition}

\newcommand{\eq}[1]{\begin{equation}\label{#1}}
\newcommand{\en}{\end{equation}}

\newenvironment{proof}{\begin{trivlist}
                       \item[]{\bf Proof.}
                       \hspace{0cm} }{\hfill $\Box$
                       \end{trivlist}}

\def\IC{\mathbb{C}}

\def\inv{^{-1}}
\def\lt{\left}
\def\rt{\right}

\newcommand{\xij}{x_j^{(i)}}
\newcommand{\rij}{r_j^{(i)}}
\newcommand{\pij}{p_j^{(i-1)}}
\newcommand{\Xei}{X^{(i)}}

\journal{Journal of Computational Physics}

\begin{document}

\begin{frontmatter}

\title{A Projected Preconditioned Conjugate Gradient Algorithm for Computing Many Extreme Eigenpairs of a Hermitian Matrix\tnoteref{mytitlenote}}
\tnotetext[mytitlenote]{This material is based upon work supported by
the U.S. Department of Energy, Office of Science, under
Scientific Discovery through Advanced Computing (SciDAC) program funded
by the Offices of Advanced Scientific Computing Research and 
Basic Energy Sciences contract number DE-AC02-05CH11231.}

%
%
%
%

\author{Eugene Vecharynski\corref{mycorrespondingauthor}}
\cortext[mycorrespondingauthor]{Corresponding author}
\ead{eugene.vecharynski@gmail.com}
\author{Chao Yang}
\address{Computational Research Division, Lawrence Berkeley National Laboratory, 1 Cyclotron Road, Berkeley, CA 94720, USA}

\author{John E.~Pask}
\address{Physics Division, Lawrence Livermore National Laboratory, 7000 East Avenue, Livermore, CA 94550, USA}

%
%

\begin{abstract}
We present an iterative algorithm for computing an invariant subspace 
associated with the algebraically smallest eigenvalues of a large sparse or 
structured Hermitian matrix $A$.  We are interested in the case in which 
the dimension of the invariant subspace is large (e.g., over several hundreds 
or thousands) even though it may still be small relative to the 
dimension of $A$. 
These problems arise from, for example, density functional theory (DFT) based 
electronic structure calculations for complex materials.  The key feature 
of our algorithm is that it performs fewer Rayleigh--Ritz calculations
compared to existing algorithms such as the locally optimal block preconditioned
conjugate gradient or the Davidson algorithm. It is a block algorithm, and hence
can take advantage of efficient BLAS3 operations and be implemented with
multiple levels of concurrency. We discuss a number of practical issues that
must be addressed in order to implement the algorithm efficiently on a
high performance computer.
\end{abstract}


\end{frontmatter}


\section{Introduction}
We are interested in efficient algorithms for computing a small percentage
of eigenpairs of a large Hermitian matrix $A$ that is
either sparse or structured (i.e., the matrix--vector product $Ax$ 
can be computed efficiently.) Often these eigenpairs correspond to the algebraically smallest eigenvalues.
This type of problem arises, for example,
in the context of Kohn-Sham density functional theory (DFT) based
electronic structure calculations of large molecules or 
solids.  

When the dimension of the matrix $n$ is above $10^6$, for example,
even half a percent of $n$ amounts to more than 5,000  eigenvalues.
When that many eigenpairs are needed, many of the existing algorithms
such as the Lanczos algorithm~\cite{Saad-book3}, the block Davidson algorithm~\cite{Davidson:75, Saad-book3}, 
which is widely used in the electronic structure calculation community, 
and the Locally Optimal Block Preconditioned Conjugate Gradient (LOBPCG) 
algorithm~\cite{Knyazev:01} are often not adequate or efficient.
This is true even when vast computational resources are available.

One of the main obstacles to achieving high performance
in the existing algorithms is the cost associated with solving
a projected eigenvalue problem whose dimension is at least as
large as the number of eigenvalues to be computed.  The solution
of this projected eigenvalue problem is part of the so-called 
Rayleigh--Ritz (RR) procedure used to extract approximations to 
the eigenpairs from a carefully constructed subspace.  
When the number of desired 
eigenpairs is large, the cost for performing this step, 
which scales cubically with respect to the dimension of the projected 
matrix, cannot be ignored. To speed up the computation, one may use 
the ScaLAPACK library~\cite{scalapack} to perform the dense eigenvalue 
calculation in parallel on a distributed memory parallel computer.  
However, for this type of calculation, it is generally difficult 
to achieve good scalability beyond a few hundred processors.

Attempts have been made in recent work to address the issue of 
high RR cost in large-scale eigenvalue computations. 
One approach is based on the idea of spectrum slicing~\cite{Schofield.Chelikowsky.Saad:12} in which the spectrum of $A$ is divided into multiple small 
intervals, and eigenvalues belonging to different intervals are computed 
simultaneously. This algorithm is potentially scalable and does not 
suffer from high RR cost.  However, dividing the spectrum in an optimal 
way is nontrivial. Furthermore, computing interior clustered eigenvalues 
can be difficult.  Another approach is based on solving the eigenvalue
problem as a penalized trace minimization~\cite{Wen.Yang.Liu.Zhang:13TR}. 
By moving the orthonormality constraint to the objective function as 
a penalty term, this scheme can use unconstrained optimization 
techniques without performing frequent RR calculations. 
However, the efficiency of the method depends on an optimal choice of
the penalty parameter, which may not be easy to obtain. The significance of reducing the 
intensity of RR calculations was pointed out in earlier works as well, e.g., by Stewart and Jennings~\cite{Stewart.Jennings:81}.

In this paper, we present an algorithm that reduces the number of 
the RR calculations. 
Our approach is similar to the Davidson-Liu and LOBPCG methods in the sense that
a preconditioned 
short-term recurrence is used to update the approximation to the 
desired 
invariant subspace.  
A key difference in the proposed scheme is
that the coefficients of the short-term recurrence are obtained by 
solving $k/q$ independent $3q \times 3q$ eigenvalue problems instead of one large 
$2k \times 2k$ or $3k \times 3k$ eigenvalue problem, where $k$ is the number of desired eigenpairs and $q$ is a chosen block size independent of $k$. 
Instead of large RR computations at every iteration, periodic basis orthogonalization is performed in the new algorithm. 
The computational kernels used in this orthogonalization step typically run more 
efficiently than dense diagonalization on high performance 
parallel computers. 

The idea of replacing the solution of a large projected eigenproblem by a sequence of smaller problems
has been considered by Knyazev in the context of the LOBPCG II algorithm~\cite{Knyazev:01}, as a means to reduce the dimension of the LOBPCG trial 
subspace from $3k$ to $k$. While the approach significantly reduces the RR cost compared to the original version of LOBPCG, its application within LOBPCG II
does not eliminate the solution of a possibly large dense eigenproblem at every iteration. Specifically, instead of solving a $3k$-by-$3k$ eigenproblem as in the
original LOBPCG method, LOBPCG II solves a $k$-by-$k$ eigenproblem, which is still costly for large $k$. 


Our approach, which we refer to as the Projected Preconditioned Conjugate Gradient (PPCG) algorithm, 
can be easily implemented by making a relatively small modification
to the existing schemes implemented in many scientific software 
packages, such as planewave DFT based electronic structure calculation
software packages.  We will show by numerical examples that PPCG
indeed outperforms the current state-of-the-art algorithms implemented in the
widely used Quantum Espresso (QE) planewave density functional electronic structure software package~\cite{QE-2009}

In this work we only consider the the standard eigenvalue problem $A x = \lambda x$. 
The generalization of the new algorithm to the case of the generalized eigenvalue problem $A x = \lambda B x$ 
is straightforward and can be performed without factoring the Hermitian positive definite matrix $B$.

The paper is organized as follows. In section~\ref{sec:tracemin}, we discuss a few 
optimization based approaches for large-scale eigenvalue computations.
We present the basic version of the PPCG algorithm in section~\ref{sec:basic}.
The connection between our approach and other available algorithms is 
discussed in section~\ref{sec:connect}. 
A number of practical aspects for implementing the new method are addressed in section~\ref{sec:practical}.  
Section~\ref{sec:example} contains numerical results. Conclusions are given in section~\ref{sec:concl}. 

\section{Trace Minimization} \label{sec:tracemin}
The invariant subspace associated with the $k$ algebraically smallest 
eigenvalues of $A$ can be computed by solving the following constrained 
optimization problem
\begin{equation}
\min_{X^*X=I} \frac{1}{2}\mbox{trace}(X^* A X),
\label{eq:tracemin}
\end{equation}
where $X\in \IC^{n \times k}$.

There exist several approaches for large-scale eigenvalue computations
that are based directly on formulation~\eqref{eq:tracemin}. 
These approaches treat the eigenvalue problem 
as the minimization problem, which allows applying relevant optimization techniques for computing the targeted 
eigenpairs. 

A number of popular algorithms for computing invariant subspaces 
are based on gradient type methods
for minimizing~\eqref{eq:tracemin}.   
In particular, projecting the gradient of the objective function in~\eqref{eq:tracemin}
along the tangent of the orthonormality constraint $X^*X=I$ 
yields the residual
\begin{equation}
R = (I-XX^\ast)AX = AX - X(X^*AX),
\label{eq:resid}
\end{equation}
which can be chosen as the search direction in an optimization 
algorithm designed to solve ~\eqref{eq:tracemin}.  
A preconditioner $T$ can be introduced to yield a modified
search direction $T R$.

In the simplest version of the Davidson-Liu algorithm, a new approximation
$\bar{X}$ is constructed by taking it to be a linear combination
of $X$ and $TR$, i.e., we write
\[
\bar{X} = X C_1 + TR C_2,
\]
where $C_1, C_2 \in \IC^{k \times k}$ are chosen to minimize
the trace of $A$ within the subspace spanned by columns of $X$ and $TR$.
The optimal $C_1$ and $C_2$ can be obtained by computing the 
lowest $k$ eigenpairs of the projected $2k \times 2k$ eigenvalue problem
\begin{equation}\label{eq:projev}
(S^* A S)C = (S^* S) C\Omega, \ \ C^* (S^*S) C = I,
\end{equation}
where $S = [ X, \ TR ]$, $C \in \IC^{2k\times k}$, and 
$\Omega \in \mathbb{R}^{k \times k}$ is a diagonal matrix that 
contains the $k$ algebraically smallest eigenvalues.  We can then 
take $C_1$ to be the first $k$ rows of $C$ and $C_2$ to contain 
the remaining rows of $C$.
The main steps of the simplest version of the Davidson-Liu algorithm 
are outlined in Algorithm~\ref{alg:davidson}.  

\begin{algorithm}[htbp]
\begin{small}
\begin{center}
  \begin{minipage}{5in}
\begin{tabular}{p{0.5in}p{4.5in}}
{\bf Input}:  &  \begin{minipage}[t]{4.0in}
The matrix $A$, a preconditioner $T$ and
                 the starting guess of the invariant subspace
                 $X^{(0)} \in \IC^{n \times k}$ associated
                 with the $k$ smallest eigenvalues of $A$, $X^{(0)*}X^{(0)} = I$;
                  \end{minipage} \\
{\bf Output}:  &  \begin{minipage}[t]{4.0in}
                 An approximate invariant subspace $\IC^{n \times k}$ 
                 associated with $k$ smallest eigenvalues of $A$;
                  \end{minipage}
\end{tabular}
\begin{algorithmic}[1]
\STATE $X \leftarrow X^{(0)}$;
\WHILE {convergence not reached} 
  \STATE $R \leftarrow T(AX - X(X^* A X))$;
  \STATE $S \gets [X, \ R]$;
   \STATE Find eigenvectors $C$ associated with the $k$ smallest eigenvalues $\Omega$ of~\eqref{eq:projev};   
  \STATE  $X \leftarrow S C$;
\ENDWHILE
\end{algorithmic}
\end{minipage}
\end{center}
\end{small}
  \caption{Simplest version of Davidson-Liu algorithm}
  \label{alg:davidson}
\end{algorithm}

Algorithm~\ref{alg:davidson} can be modified to accumulatively include multiple $TR$ 
blocks computed from different iterations in $S$, which gives the conventional Davidson  method~\cite{Davidson:75}.
Alternatively, the algorithm's iterations can be altered to include the so-called ``conjugate'' direction $P$
as part of the search space $S$, 
i.e., one can let
$S \leftarrow [X, TR, P]$ and solve a $3k \times 3k$
projected eigenvalue problem~\eqref{eq:projev}.  
The block $P$ can be constructed as a linear combination of 
$TR$ and $P$ computed at the previous iteration. 
Such a modification leads to the locally optimal block preconditioned 
conjugate gradient (LOBPCG) algorithm originally proposed in~\cite{Knyazev:01}.

When a good preconditioner $T$ is available, as is the case 
for planewave based electronic structure calculations, 
both the simplest version of the Davidson-Liu algorithm and 
the LOBPCG algorithm can converge rapidly. The number of iterations 
required by LOBPCG to reach convergence is often smaller than that taken 
by the Davidson-Liu algorithm, but each Davidson iteration is
slightly cheaper because it solves a $2k \times 2k$ instead
of a $3k \times 3k$ projected eigenvalue problem. 
When $k$ is relatively small, such extra cost per iteration is
negligible. However, when $k$ is relatively large (e.g., 
on the order of thousands or more) the cost of solving the
projected eigenvalue problem, which we refer to as the 
RR cost, can no longer be ignored.

Although the RR eigenvalue problem can be solved in
parallel using the ScaLAPACK library, the performance
of this part of the calculation generally does not 
scale well beyond a few hundred cores.  Although some
progress has recently been made on speeding up symmetric dense
eigenvalue calculation on distributed memory parallel 
computers \cite{ELPA, ELEMENTAL}, the performance of the latest algorithms still
lags behind that of level three BLAS and other computational 
building blocks of electronic structure codes.

Another optimization based approach for eigenvalue computations 
was proposed by Sameh and Wisniewski~\cite{sw1982,st2000}. Their TRACEMIN algorithm is different from the 
gradient type schemes applied to~\eqref{eq:tracemin}. 
It relies on the trace minimization procedure, which solves  
a sequence of correction problems of the form  
\begin{equation}
\min_{X^\ast \Delta = 0} \mbox{trace}(X-\Delta)^\ast A (X-\Delta).
\label{eq:trmin_corr}
\end{equation}
The solution of~\eqref{eq:trmin_corr} is obtained 
by iteratively solving the projected linear system
\[
(M A M) \Delta = MAX, \ \ X^{\ast} \Delta  = 0,
\]
where $M = I-XX^{\ast}$. A RR procedure is then performed 
within the subspace spanned by the columns of $X-\Delta$ in each step to 
produce a new approximation to the solution of~\eqref{eq:tracemin}.

\section{The Projected Preconditioned Conjugate Gradient algorithm} 
\label{sec:basic}

In this section, we present a preconditioned conjugate gradient type of 
scheme to find a solution of the minimization problem~\eqref{eq:tracemin}. 
The proposed approach is motivated by the gradient projection techniques for constrained optimization (e.g.,~\cite{Haftka.Gurdal:92, Nocedal.Wright:99, Polyak_book_eng:87}).  

Given a function $f(x)$ whose minimum is sought over a set $Q$ defined by constraints, 
the general framework of gradient projection methods is to iteratively perform a sequence of updates
$\bar x \gets  x + \gamma s$, where the updated approximation $\bar x$ is allowed to leave the set $Q$ that represents 
feasible regions.  
The new approximation, however, is then projected back to the feasible set
$Q$, i.e., the new iterate $x$ is defined 
as $x \gets M_Q \bar x$, where $M_Q$ is an appropriately 
defined projector onto $Q$. 

The search direction $s$ can be defined in a number of ways. 
For example, 
it can be chosen as the gradient $\nabla f(x)$ of $f$ evaluated at the current approximation 
$x$~\cite{Levitin.Polyak:66,  Goldstein:64}. In equality 
constrained optimization, the search direction is often taken to be the 
projection of the gradient onto the tangent of constraints, i.e.,
$s = M \nabla f(x)$, where $M$ is a corresponding projection 
defined in terms of the normal of the equality constraint 
that implicitly defines the region $Q$ in which $x$ must lie~\cite{Rosen1:60, Rosen2:61}.  
In particular, for the equality constraint $X^*X = I$ in~\eqref{eq:tracemin},
$M = I-XX^{\ast}$.

We consider an extension of the gradient projection approach to 
trace minimization~\eqref{eq:tracemin} in which the approximation
to the 
minimizer is updated as follows:
\begin{equation}\label{eq:ppcg}
\bar X \gets X C_X + W C_W + P C_P, \quad X \gets M_Q \bar X, \quad \end{equation}   
where the search direction $W = (I - X X^*) T R$ is given by the preconditioned residual $T R = T(AX - X (X^* A X))$ 
projected onto the tangent of the orthonormality constraint and
$P = (I - X X^*) (W^{\prime} C_W^{\prime} - P^{\prime} C_P^{\prime})$ 
represents a conjugate direction in the same tangent space. The ``prime notation''
refers to the corresponding quantities from the previous step. 

Extracting the best approximation 
from the subspace spanned by the columns of $X$, $W$ and $P$ (as traditionally done) would   require 
a RR calculation that is costly when the number of desired
eigenpairs is large.  To reduce such cost, we relax the optimality requirement on the search 
parameters $C_X$, $C_W$, and $C_P$ in~\eqref{eq:ppcg}, and allow them to introduce non-orthogonality 
in the updated columns of $\bar X$. This places $\bar X$ outside of the feasible region given by the orthogonality
constraint, which is remedied by the subsequent application of a projector $M_Q$.

Specifically,    
let us restrict $C_X$, $C_W$, and $C_P$ to be diagonal matrices. 
An advantage of such a restriction is that the iteration parameters associated
with each column of the updated $X$ can be determined independently. 
%
More generally, it is possible to allow $C_X$, 
$C_W$ and $C_P$ to be block diagonal matrices 
with small diagonal blocks (e.g., $5\times 5$ or $10 \times 10$ blocks). One can then expect that, if properly chosen, the extra degrees of 
freedom introduced by these diagonal blocks can reduce the iteration count.
The general block diagonal formulation will be discussed in section~\ref{subsec:block}.

Let $C_X = \text{diag} \{ \alpha_1, \ldots,\alpha_k \}$, $C_W = \text{diag}\{ \beta_1, \ldots,\beta_k \}$, 
and $C_P = \text{diag}\{ \gamma_1, \ldots,\gamma_k \}$.
Then iteration~\eqref{eq:ppcg} gives a sequence of $k$ single-vector 
updates 
\begin{equation}\label{eq:upd_col}
\bar x_j \gets \alpha_j x_j + \beta_j w_j +  \gamma_j p_j, \quad j = 1, \ldots,k; 
\end{equation}
where $\bar x_j$, $x_j$, $w_j$, and $p_j$ denote the $j$th columns of $\bar X$, $X$, $W$, and $P$, respectively. 
Let us choose $\alpha_j$, $\beta_j$, and $\gamma_j$ in such a way that
each corresponding updated column $\bar x_j$ yields the minimizer of $x^* A x$, subject to 
the normalization constraint $\|x\| = 1$, over the corresponding subspace 
spanned by $x_j$, $w_j$, and $p_j$.  
Clearly, the computations of the parameter triplets are independent of each other, and can be
performed by solving $k$ separate $3$-by-$3$ eigenvalue problems. 

As a result of the decoupled steps~\eqref{eq:upd_col}, the columns $\bar x_j$ are generally not orthogonal to each other. 
Moreover, 
they can all converge to the same eigenvector associated with the smallest 
eigenvalue of~$A$ without any safeguard in the algorithm. To overcome this issue, we project the updated block $\bar X$ 
back onto the orthonormality constraint $X^* X = I$ by 
performing a QR factorization of $X$ and setting the new approximation $X$ to the obtained orthogonal factor.
This step corresponds to the action of applying the projector $M_Q$ in~\eqref{eq:ppcg}, which we  
further denote by $X \leftarrow  \texttt{orth}(\bar X)$. 

\begin{algorithm}[htbp]
\begin{small}
\begin{center}
  \begin{minipage}{5in}
\begin{tabular}{p{0.5in}p{4.5in}}
{\bf Input}:  &  \begin{minipage}[t]{4.0in}
The matrix $A$, a preconditioner $T$, and a starting guess of the invariant subspace $X^{(0)} \in \IC^{n \times k}$ 
associated with the $k$ smallest eigenvalues of $A$; 
                  \end{minipage} \\
{\bf Output}:  &  \begin{minipage}[t]{4.0in}
                 An approximate invariant subspace $X \in \IC^{n \times k}$ 
                 associated with the $k$ smallest eigenvalues of $A$;
                  \end{minipage}
\end{tabular}
\begin{algorithmic}[1]
\STATE $X \gets \texttt{orth} (X^{(0)})$; $P \gets \lt[ \ \rt]$;
\WHILE {convergence not reached}
  \STATE $W \gets T (AX - X(X^* A X))$;
  \STATE $W \gets (I - XX^*)W$;
  \STATE $P \gets (I - XX^*)P$;
  \FOR {$j = 1, \ldots, k$} 
      \STATE $S \gets [x_j , w_j , p_j]$;
      \STATE Find the smallest eigenpair ($\theta_{\min}, c_{\min}$) of $S^*AS c = \theta S^* Sc$, where $c^* S^*S c = 1$; 
      \STATE $\alpha_j \gets c_{\min}(1)$, $\beta_j \gets c_{\min}(2)$; and  $\gamma_j \gets c_{\min}(3)$ ($\gamma_j = 0$ 
at the initial step);
      \STATE $p_j \gets \beta_j w_j + \gamma_j p_j$;
      \STATE $x_{j} \gets \alpha_j x_j + p_j$. 
  \ENDFOR
  \STATE $X \gets \texttt{orth}(X)$;
  \STATE If needed, perform the Rayleigh-Ritz procedure within $\mbox{span}(X)$;
\ENDWHILE
\end{algorithmic}
\end{minipage}
\end{center}
\end{small}
  \caption{The projected preconditioned conjugate gradient (PPCG) algorithm}
  \label{alg:ppcg0}
\end{algorithm}

The proposed approach is outlined in 
Algorithm~\ref{alg:ppcg0} that we refer to as the Projected
Preconditioned Conjugate Gradient (PPCG) algorithm.
Note that in practical implementations we require the method to perform 
the RR procedure every once in a while (step 14). 
Such periodic RR computations allow ``rotating'' the columns of $X$ closer to the targeted eigenvectors. 
They also provide opportunities for us to identify converged eigenvectors and 
deflate them through a locking mechanism 
discussed in
section~\ref{sec:practical}.
In our experiments, we typically perform an RR calculation every 5-10 iterations,
which is significantly less frequent compared to the existing eigensolvers that 
perform the RR procedure at each step.   

In principle, the RR procedure in step 14 of Algorithm~\ref{alg:ppcg0} can be omitted.  
In this case, the columns of the iterates $X$ generally do not converge to 
eigenvectors, i.e.,
$X$ only represents some orthonormal basis of the approximate invariant subspace.
However, in a number of our test problems, 
the absence of step 14 led to the convergence deterioration. Therefore, performing a periodic RR step can be helpful to
ensure the eigensolver's robustness. We will return to this discussion in section~\ref{sec:practical}.

An important element of Algorithm~\ref{alg:ppcg0} is the orthonormalization of the block $X$ in step~13. It is clear that 
the $\mbox{orth}(X)$ procedure is well-defined if and only if it is applied to a full-rank matrix, i.e., if and only if the 
single-vector sweep in Steps 6-12 yields a block $X$ of linearly independent vectors. 
The following theorem shows that the linear independence among 
columns of $X$ is guaranteed if all the parameters $\alpha_j$ are nonzero. 

\begin{theorem}\label{thm:1}
Let vectors $\bar x_j$ be computed according to~\eqref{eq:upd_col}, where $x_j$, $w_j$, and $p_j$ denote the $j$th columns of $X$, $W$, and $P$, respectively; and let $X^*X = I$. 
Then the matrix $\bar X = [\bar x_1, \ldots, \bar x_k]$ is full-rank if $\alpha_j \neq 0$ for all $j$.
\end{theorem}
\begin{proof}
Let us write~\eqref{eq:upd_col} in the matrix form as
\begin{equation}\label{eq:expandx}
\bar X  = X C_X + K, 
\end{equation}
where $K = W C_W + P C_P$; with $C_X$,  $C_W$,  $C_P$ being the diagonal matrices of iteration coefficients 
$\alpha_j$, $\beta_j$, and $\gamma_j$, respectively.
We assume, on the contrary, that columns of $\bar{X}$ are linearly dependent.
Then there exists a vector $y\neq 0$ such that $\bar{X}y=0$. It follows from~\eqref{eq:expandx}
and the conditions $X^* X=I$ and $X^{*}K=0$ that $C_Xy=0$. Since $C_X$
is diagonal, then at least one of its diagonal elements
must be zero. 
This contradicts the assumption that  $\alpha_j \neq 0$ for all $j$.  
%
\end{proof}

Theorem~\ref{thm:1} provides us with a simple indicator of rank deficiency in 
$\bar{X}$. In the case when $\bar{X}$ becomes rank deficient, there is a simple
way to fix the problem. We can simply backtrack and exclude the $P$ block 
in \eqref{eq:ppcg} and take a steepest descent-like step by recomputing 
the coefficients in $C_X$ and $C_W$.

The next theorem shows that, in this case, if the preconditioner $T$ is Hermitian positive 
definite (HPD), the recomputed $\alpha_j$'s are guaranteed to be nonzero and 
therefore the updated $X$ is full rank, unless some columns of $W$ become
zero, which indicate the convergence of some eigenvectors that should be
deflated at an earlier stage.

\begin{theorem}\label{thm:2}
Let $P \equiv 0$ in \eqref{eq:ppcg} 
so that the update in $X$ are computed as 
\begin{equation}\label{eq:upd_col_trunc}
\bar x_j \gets \alpha_j x_j + \beta_j w_j, \quad j = 1, \ldots,k.
\end{equation}
If the preconditioner $T$ is~HPD and the $j$th column $r_j$ of the residual 
$AX - X(X^*AX)$ is nonzero, then $\bar X = [\bar x_1, \ldots, \bar x_k]$ 
must have a full rank. 
\end{theorem}
\begin{proof} It follows from Theorem 1 that if $\bar{X}$ is rank deficient, 
then there is at least one $j$ such that $\alpha_j = 0$. Since 
$\alpha_j$ is the first component of an eigenvector $c = (\alpha_j, \beta_j)^*$ of the 2-by-2 projected eigenproblem 
\begin{equation}\label{eq:proj2}
\left(
\begin{array}{cc}
x_j^*A x_j & x_j^* A w_j \\
x_j^*A w_j & w_j^* A w_j 
\end{array}
\right) c = \theta 
\left(
\begin{array}{cc}
x_j^* x_j & 0 \\
0 & w_j^* w_j 
\end{array} 
\right) c,
\end{equation}
the coefficient $\beta_j$ must be nonzero, and the matrix on the left hand 
side of \eqref{eq:proj2} must be diagonal, i.e., $x_j^* A w_j = 0$. 
However, since
\[
x_j^*Aw_j =  (Xe_j)^*A (I-XX^*) T \left[(I-XX^*)AX\right]e_j = r_j^* T r_j,
\]
$r_j$ must be zero since $T$ is assumed to be HPD. This contradicts the 
assumption that $r_j$ is nonzero. Hence, it follows that $\alpha_j \neq 0$, 
and $\bar{X}$ must be full~rank.
\end{proof}


We note, however, that the situation where $\alpha_j$ is zero is very unlikely in practice
and, in particular, has never occurred in our numerical tests.   


\section{Relation to other algorithms}\label{sec:connect}
The PPCG algorithm is similar to a variant of the LOBPCG 
algorithm called LOBPCG II presented in~\cite{Knyazev:01}, 
in the sense that it 
updates 
the subspace via $k$ independent $3 \times 3$ RR computations rather than one large $3k \times 3k$ one, as in LOBPCG. 
As discussed in the next section, the PPCG algorithm generalizes this to $3q \times 3q$ RR computations, where $q$ is a chosen block size, to increase convergence rate and better exploit available computational kernels. 
The main difference, however, is that in LOBPCG II, 
the RR procedure is performed within the subspace
spanned by the columns of $X$ in each iteration, 
whereas in PPCG, $X$ is merely orthogonalized, 
and RR computations are invoked only periodically
(e.g., every 5 or 10 iterations). Another difference is
related to the construction of the blocks $W$ and $P$. In contrast
to the proposed PPCG algorithm, LOBPCG II does not carry out the orthogonalizations
of the preconditioned residuals and conjugate directions against the
approximate invariant subspace $X$.
Furthermore, to allow the replacement of RR with periodic orthogonalization, the definition of the PPCG residuals has been generalized to the case in which $X$ is not necessarily formed by a basis of Ritz vectors and
the matrix $X^*AX$ is not necessarily diagonal. 
As a result, the separate minimizations in
PPCG and LOBPCG II are performed with respect to different 
subspaces, and the methods
are not equivalent even if the PPCG block size is $1$.

The TRACEMIN algorithm~\cite{sw1982} also becomes similar to the PPCG algorithm 
if the RR procedure is performed periodically. Instead of minimizing 
several Rayleigh quotients, TRACEMIN solves several
linear equations using a standard preconditioned conjugate gradient (PCG)
algorithm. Typically, more than one PCG iteration is needed to
obtain an approximate solution to each equation.


The PPCG method can also be viewed as a compromise between
a full block minimization method such as the LOBPCG method,
which converges rapidly but has a higher RR cost per iteration,
and a single vector method combined with an appropriate 
deflation scheme (also known as a band-by-band method), which 
has a negligible RR cost but slower overall convergence rate because
one eigenpair is computed at a time. Also, band-by-band methods
cannot effectively exploit the concurrency available in multiplying
$A$ with a block of vectors, and hence are often slower in practice 
on high performance parallel computers.

\section{Practical aspects of the PPCG algorithm}\label{sec:practical}

In this section, we address several practical aspects of the PPCG algorithm 
that are crucial for achieving high performance. We first consider the 
generalization of the single-vector updates in~\eqref{eq:upd_col} 
to block updates.

\subsection{Block formulation}\label{subsec:block}

As indicated earlier, we can allow $C_X$, $C_W$ and 
$C_P$ to be block diagonal. 
In this case,
$C_X = \mbox{diag}\{ C_{X_1}, \ldots, C_{X_s}\}$, 
$C_W = \mbox{diag}\{ C_{W_1}, \ldots, C_{W_s}\}$, 
and $C_P = \mbox{diag}\{ C_{P_1}, \ldots, C_{P_s}\}$,
and $X$, $W$ and $P$ can be partitioned conformally as 
$X = [X_1, X_2, \ldots , X_s]$, 
$W = [W_1, W_2, \ldots , W_s]$, and 
$P = [P_1, P_2, \ldots , P_s]$, 
where the $j$th subblocks of $X$, $W$, and $P$ contain $k_j$ columns, and
$\sum_{j=1}^s k_j = k $. 

The single-column sweep~\eqref{eq:upd_col} is then replaced by block updates
\begin{equation}\label{eq:upd_block}
\bar X_j \gets X_j C_{X_j} + W_j C_{W_j} + P_j C_{P_j}, \quad j = 1,\ldots,s. 
\end{equation}
After columns of $\bar X = [\bar X_1, \bar X_2, \ldots , \bar X_s]$ 
are orthonormalized, we obtain a new approximation which is used as a 
starting point for the next PPCG iteration. For each $j$, the block
coefficients $C_{X_j}$, $C_{W_j}$, and $C_{P_j}$
in~\eqref{eq:upd_block} are chosen to minimize the trace~\eqref{eq:tracemin} 
within span$\{X_j, W_j, P_j\}$. This is equivalent to 
computing the $k_j$ smallest eigenvalues and corresponding eigenvectors
of $3 k_j$-by-$3 k_j$ eigenvalue problems~\eqref{eq:projev} with 
$S = [X_j, W_j, P_j]$.  Thus, the block formulation of the PPCG algorithm performs
$s$ iterations of the ``for'' loop in lines 6-12 of Algorithm~\ref{alg:ppcg0}.

Note that, in the extreme case where $s = 1$, i.e., the splitting corresponds to the whole block, the PPCG algorithm becomes equivalent to 
LOBPCG~\cite{Knyazev:01}. In this case the matrices $C_X$, $C_W$, 
and $C_P$ are full and generally dense, and the updated 
solution is optimal within the subspace spanned by all columns of $X$, $W$, and $P$.

As we will demonstrate in section~\ref{sec:example}, making
$C_X$, $C_W$ and $C_P$ block diagonal
generally leads to a reduction in the number of outer iterations required
to reach convergence. However, as the block size increases, the 
cost associated with solving $s$ $3k_j \times 3k_j$ eigenvalue 
problems also increases.  The optimal choice of $k_j$ will be problem- and computational platform 
dependent. Heuristics must be developed to set $k_j$ to an appropriate
value in an efficient implementation.

In our current implementation, we set $k_j$ to a constant $sbsize$ with
the exception that the last block of $X$, $W$, and $P$ may contain a 
slightly different number of columns.   
In principle, one can choose different $k_j$ values for each subblock. For example, this could be helpful if additional information about the distribution of $A$'s spectrum is available.
In this case, a proper uneven splitting could potentially allow for a better resolution of eigenvalue clusters. 

\subsection{Convergence criteria}
We now discuss appropriate convergence criteria for terminating
the PPCG algorithm.  If, instead of individual eigenpairs, we are only interested in the invariant
subspace associated with the smallest eigenvalues of $A$, 
we may use the following relative subspace 
residual norm 
\begin{equation}
\frac{\| AX - X(X^\ast AX) \|_F} {\| X^\ast A X\|_F},
\label{eq:resnrm}
\end{equation}
as a metric to determine when to terminate the PPCG iteration, where $\|\cdot\|_F$ is the Frobenius
norm.
No additional multiplication of $A$ with $X$ is required
in the residual calculation.  Checking the subspace 
residual does not require a RR calculation. 

Since the above measure monitors the quality of the whole approximate invariant subspace,
it does not allow one
to see whether the 
subspace 
contains good approximations to some of the eigenvectors
that can be locked and deflated from subsequent 
PPCG iterations.  This is a reason periodic
RR can be helpful. We will discuss 
deflation in section~\ref{sec:deflation}.

In some cases, especially in the early PPCG iterations in which 
significant changes in $X$ can be observed, it may not be 
necessary to check the subspace residual.  Since the objective of 
the algorithm is to minimize the trace of $X^*AX$, it is reasonable 
to use the relative change in the trace, which can be computed  
quickly, as a measure for terminating the PPCG iteration.
To be specific, if $X$ and $X^{\prime}$ are approximations 
to the desired invariant subspace obtained at the current and previous
iterations, we can use 
\[
\frac{\left | \mathrm{trace}\left( X^* A X \right) - 
\mathrm{trace}\left(X^{\prime *}A X^{\prime} \right) \right |}
{\mathrm{trace}\left(X^* A X\right)}
< \tau,
\]
as a criterion for terminating the PPCG iteration, where
$\tau$ is an appropriately chosen tolerance. 
This criterion is often used in an iterative eigenvalue
calculation called within each self-consistent 
field iteration for solving the Kohn--Sham nonlinear
eigenvalue problem~\cite{Hohenberg.Kohn:64, Kohn.Sham:65}.

\subsection{Buffer vectors}

When the $k$th eigenvalue is not well separated from the $(k+1)$st 
eigenvalue of $A$, the convergence of that eigenvalue may be 
slow in subspace-projection based solution methods, see, e.g., \cite{Knyazev:01}. 

As noted in~\cite{Kn.Ar.La.Ov:07}, one way to overcome this, is to expand the block $X$ 
with $l$ additional columns $Y_l$ (a standard approach in electronic structure calculations) which we call \textit{buffer vectors}. In this case,
we set $X \gets [X, Y_l]$ and
apply Algorithm~\ref{alg:ppcg0} to the extended block with $k' = k+l$ columns. The main difference is that 
one has to monitor the convergence only to the invariant subspace that is associated with the $k$ wanted
eigenvalues, i.e., only the initial $k$ columns of the expanded $X$ should be used to evaluate the convergence
metrics discussed in the previous section.   

It is clear that introducing buffer vectors increases the cost of the algorithm, per iteration. 
For example, the cost of matrix--block multiplications with $A$ becomes higher, more work is required to perform dense
linear algebra (BLAS3) operations, the number of iterations of the inner ``for'' loop in lines 6-12 increases to $k' = k+l$ steps, etc. 
However, the number of buffer vectors $l$ is normally chosen to be small relative to $k$, e.g., $1-5\%$ of the number of targeted eigenpairs,
and the increase in the computational work per iteration is relatively small, while the decrease in iterations required to reach convergence can be substantial.

\subsection{Orthogonal projection of the search direction}

The projector $I - XX^\ast$ in the definition of search directions $W$ and $P$ turns out to be crucial for achieving rapid
convergence of the PPCG algorithm. This is in contrast to the LOBPCG algorithm,  
where applying $I - XX^\ast$ is not as important,
at least not in exact arithmetic, because a new approximation
to the desired invariant subspace is constructed from the subspace
spanned by the columns of $X$, $W$ and $P$. The use of $I-XX^\ast$ 
in the construction of $W$ and $P$ does not change that subspace.
However, in PPCG, application of $I - XX^\ast$ affects the low-dimensional 
subspaces spanned by individual columns of $X$, $W$, and $P$. 


\begin{figure}[htbp]
\begin{center}%
    \includegraphics[width=6cm]{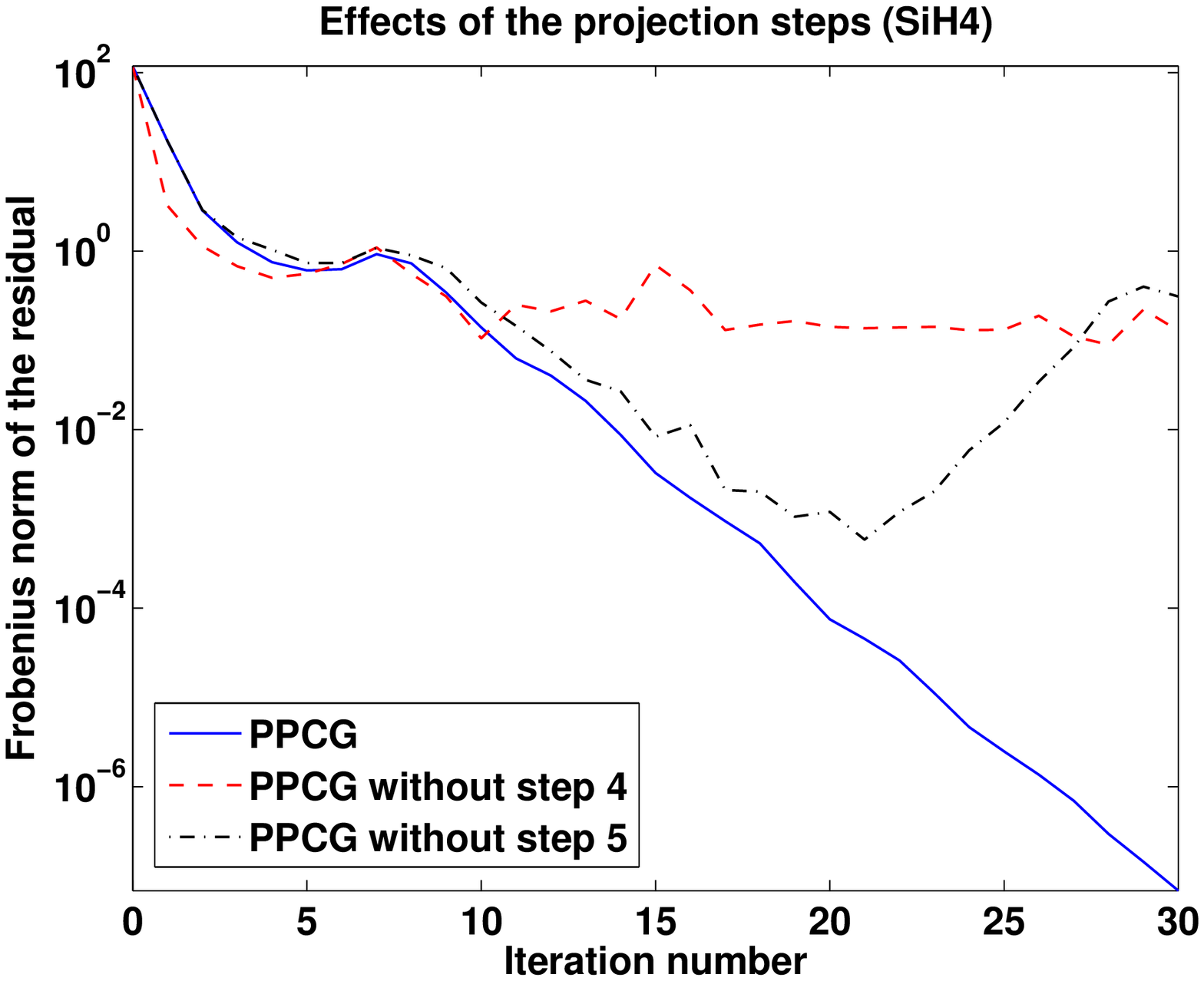}
    \includegraphics[width=6cm]{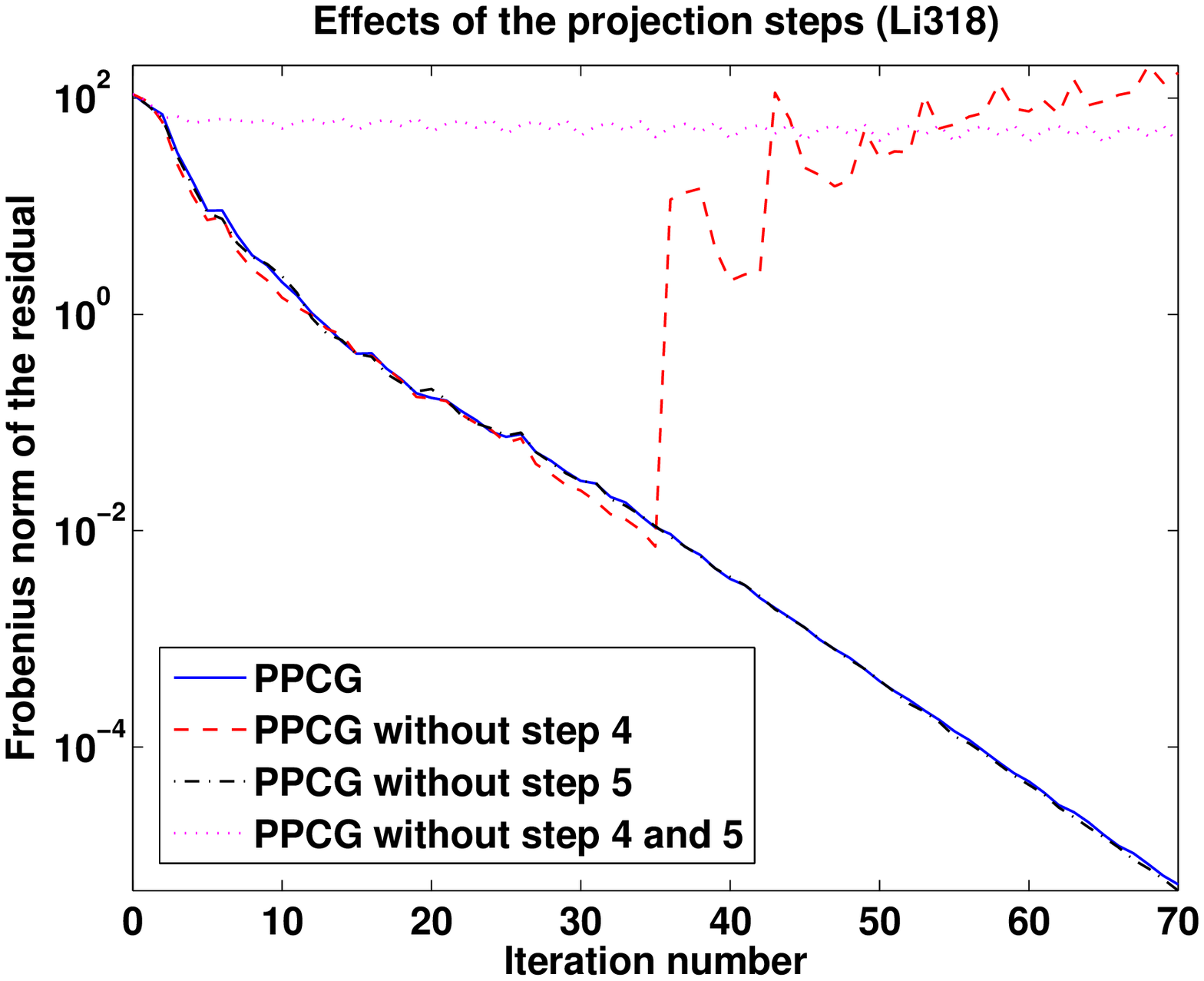}
\end{center}
\caption{Effects of the projection steps 4 and 5 of Algorithm~\ref{alg:ppcg0} 
on convergence. The PPCG variants are applied to compute 10 (left) and 2,000 (right) lowest eigenpairs of the 
converged Kohn-Sham Hamiltonian of the silane molecule (left) and the Li318 lithium-ion electrolyte (right). 
}\label{fig:proj} 
\end{figure}


As demonstrated in Figure~\ref{fig:proj}, the PPCG algorithm can be very sensitive to the orthogonal projector used in steps 4 and 5 of Algorithm~\ref{alg:ppcg0}. We observe that removing either of the two projection steps can lead to 
a severe deterioration of convergence. 
In Figure~\ref{fig:proj} (left) the PPCG algorithm is used to compute the 10 
lowest eigenpairs of a Kohn-Sham Hamiltonian of the SiH4 (silane) molecule. 
In Figure~\ref{fig:proj} (right), a similar computation is performed 
for the Li318 (lithium-ion electrolyte) system, where $2,000$ eigenpairs are sought. 

Interestingly, we observed that the effects of applying $I - XX^\ast$ are more pronounced in the cases where $A$ has multiple
eigenvalues. Furthermore, in some experiments, we noticed that skipping the projector only in $P$ may not alter the convergence; see Figure~\ref{fig:proj} (right). 
Nevertheless, we recommend keeping $(I - XX^*)$ for computing both $W$ and $P$
to achieve robust convergence. 

\subsection{Orthogonalization of the approximate invariant subspace}\label{subsec:qr}
There are a number of ways to obtain an orthonormal basis of $X$ after
its columns have been updated by the ``for'' loop in lines 6-12 of
Algorithm~\ref{alg:ppcg0}. If the columns of $X$ are far from 
being linearly dependent, an orthonormalization procedure based on using
the Cholesky factorization of $X^{\ast} X = R^{\ast}R$, where $R$ 
is a unit upper triangular matrix, is generally efficient. In this case, 
$X$ is orthonormalized by 
\[
X \leftarrow X R^{-1},
\]
i.e., step 13 of Algorithm~\ref{alg:ppcg0} is given by the 
QR decomposition based on the Cholesky decomposition (the Cholesky QR factorization). 

If the columns of $X$ are almost orthonormal, as expected when
$X$ is near the solution to the trace minimization problem, we may 
compute the orthonormal basis as
$
X \gets X (X^* X)^{-1/2}
$,  
where $(X^* X)^{-1/2} = (I + Y)^{-1/2}$ can be effectively approximated by several
terms in the Taylor expansion of $f(x) = \sqrt{1 + y}$. This gives the 
following orthogonalization procedure:
\[
X \leftarrow X (I - Y/2 + 3Y^2/8 - 5Y^3/16 + \cdots), \quad Y = X^* X -  I.
\]
Since $Y$ is likely to be small in norm, we may only need three or four terms 
in the above expansion to obtain a nearly orthonormal $X$.
An attractive feature of this 
updating scheme is that it uses only dense matrix--matrix multiplications 
which can be performed efficiently on modern high performance 
computers. Note that the updated matrix represents an orthonormal factor in the polar
decomposition~\cite{Horn.Johnson:90} of $X$, which gives a matrix with orthonormal columns that is closest to $X$~\cite{Fan.Hoffman:55}.

\begin{figure}[htbp]
\begin{center}%
    \includegraphics[width=6cm]{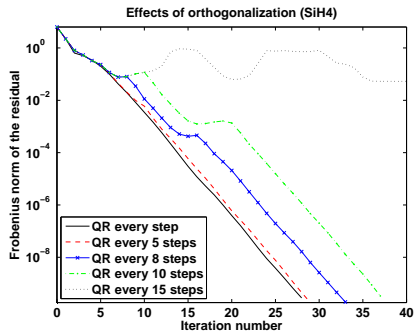}
\end{center}
\caption{Effects of removing the orthonormalization step 13 of 
Algorithm~\ref{alg:ppcg0}, when computing the $4$ lowest 
eigenpairs of a Kohn-Sham Hamiltonian associated with the silane molecule. 
The problem size is $2,103$.}\label{fig:orth_skip} 
\end{figure}

We have observed that columns of $X$ often remain nearly orthonormal
after they are updated according to \eqref{eq:upd_block}.  
Motivated by this observation, we experimented with performing
orthonormalization periodically (step 13 of Algorithm~\ref{alg:ppcg0}) 
in the PPCG iteration. Figure~\ref{fig:orth_skip} demonstrates the effects 
of this strategy on the convergence of the algorithm.
The plotted convergence curves correspond to PPCG runs in which
orthonormalization is performed every $t$ steps, where $t = 1, 5, 8, 10, 15$. 
In our implementation, we use the Cholesky QR factorization to 
orthonormalize the columns of $X$.

As can be seen in Figure~\ref{fig:orth_skip}, skipping the QR factorization 
of $X$ for a small number of PPCG iterations (up to 5 steps in this example) 
does not affect the convergence of the algorithm. In this case, the loss of
orthogonality, which can be measured by $\|X^* X - I\|_F$, is 
at most $O(10^{-1})$. Note that when the columns of
$X$ are not orthonormal, the projectors in steps 4 and 5 of 
Algorithm~\ref{alg:ppcg0} become approximate projectors.
Nevertheless, the convergence of PPCG is not substantially affected as
long as $t$ is not too large.   
However, if we reduce the frequency of the QR factorizations significantly, 
the number of PPCG iterations required to reach convergence starts to increase. 
For this example, when we perform the QR factorization every $15$ iterations, 
PPCG fails to converge (within an error tolerance of $10^{-2}$) within 40 
iterations. In this case, the loss of orthogonality in $X$ reaches 
$O(1)$, which severely affects the convergence of the method.
 
Thus, in order to gain extra computational savings, we can devise an 
optional heuristic based on the measured loss of orthogonality. 
We can decide to skip step 13 of Algorithm~\ref{alg:ppcg0} if the 
loss of orthogonality is relatively small. 


\subsection{Periodic RR computation}
For many problems, the orthogonal projection of $W$ and $P$ against 
columns of $X$ and subsequent orthogonalization are not enough to ensure that $X$ converges rapidly 
to a basis of the desired invariant subspace.  
%
We found that a practical remedy for avoiding possible convergence 
degradation or failure is to perform the RR procedure periodically, 
which has the effect of systematically repositioning the $j$th column 
of $X$ towards the eigenvector associated with the $j$th eigenvalue of $A$. 
In this case, since 
each column of $X$ is forced to be sufficiently close to an eigenvector, minimizing 
$k$ Rayleigh quotients separately 
becomes
just as effective
as minimizing the trace of $X^*AX$ under the orthonormality constraint.
Therefore, in practical implementations, we perform RR periodically,
even though PPCG has been observed to converge without performing this step for some problems.     
This is done in step 14 of Algorithm~\ref{alg:ppcg0}. 

Another reason periodic RR may be advantageous is 
that it provides an opportunity to lock converged eigenvectors 
and reduce the number of sparse matrix vector multiplications
required to find the remaining unconverged eigenvectors.

Clearly, introducing the periodic RR calculation
increases the cost of some PPCG iterations and can potentially 
make the algorithm less scalable due to the lack of scalability 
of the dense eigensolver. However, the extra cost can be
offset by accelerated convergence of the algorithm and reduced number of sparse 
matrix vector multiplications once some approximate eigenvectors
have converged.
In our PPCG implementation, 
we control the frequency of the RR calls by a parameter 
\textit{rr\_period}.  Our numerical experiments suggest that a 
good value of \textit{rr\_period} is between 5 and 10.
Because good approximations to desired eigenvectors do not 
emerge in the first few PPCG iteration, \textit{rr\_period} 
can be set to a relatively large value and then 
decreased in later iterations when many converged eigenvector
approximations can be found and locked.


%

\subsection{Locking converged eigenvectors} \label{sec:deflation}
Even before the norm of the subspace residual $R \equiv AX - X(X^{\ast} A X)$ 
becomes small, some of the Ritz vectors associated with the subspace 
spanned by columns of $X$ can become accurate. However, these Ritz 
vectors generally do not reveal themselves through the norm of each 
column of $R$ because each column of $X$ may consist of a linear
combination of converged and unconverged eigenvector approximations.
The converged eigenvector approximations 
can only be revealed through the RR procedure.
As mentioned earlier, by performing an RR calculation, we rotate 
the columns of $X$ to the Ritz vectors which are 
used as
starting points for independent Rayleigh quotient minimization carried out
in the inner loop of the next PPCG iteration.

Once the converged Ritz vectors are detected, we lock these 
vectors by keeping them in the leading columns of $X$.
These locked vectors are not updated in inner loop of the PPCG 
algorithm until the next RR procedure is performed.  We do not 
need to keep the corresponding columns in the $W$ and $P$ matrices. 
However, the remaining columns in $W$ and $P$ must be
orthogonalized against all columns of $X$. This ``soft locking'' strategy is along the lines of that described in~\cite{Kn.Ar.La.Ov:07}. 


%
A detailed description of the PPCG algorithm, incorporating the above practical aspects, is given in Algorithm~\ref{alg:ppcg} of Appendix A.

\section{Numerical examples}\label{sec:example}
In this section, we give a few examples of how the PPCG algorithm
can be used to accelerate both the SCF iteration and the band structure
analysis in Kohn-Sham DFT calculations. We implement PPCG within 
the widely used Quantum Espresso (QE) planewave pseudopotential density functional electronic structure code~\cite{QE-2009}, and compare the performance 
of PPCG with that of the state-of-the-art Davidson solver implemented in QE. The QE code 
also contains an implementation of a band-by-band conjugate 
gradient solver.  Its performance generally lags behind that of 
the Davidson solver, however, especially when a large number of eigenpairs 
is needed. Therefore, we compare to the Davidson solver here.

In QE, the Davidson algorithm can construct a subspace $Y$ of dimension
up to $4k$ before it is restarted.  However, when the number of desired 
eigenpairs $k$ is large, solving a $4k \times 4k$ projected eigenvalue 
problem is very costly. Therefore, in our tests, we limit the subspace dimension 
of $Y$ to $2k$. 

The problems that we use to test the new algorithm are listed in 
Table~\ref{tab:probs}. A sufficiently large supercell is used 
in each case so that we perform all calculations at the $\Gamma$-point
only.  The kinetic energy cutoff (ecut), which determines the 
the number of planewave coefficients ($n_G$), as well as the 
number of atoms ($n_a$) for each system are shown. 
Norm conserving pseudopotentials are used to construct the
Kohn-Sham Hamiltonian. 
Therefore, all eigenvalue problems we solve
here are standard eigenvalue problems, although our algorithm can be
easily modified to solve generalized eigenvalue problems.
The local density approximation (LDA)~\cite{Kohn.Sham:65} is used for 
exchange and correlation. The particular choice of  
pseudopotential and exchange-correlation functional is not 
important here, however, since we focus only on the performance of the 
eigensolver.   

The distribution of eigenvalues for each problem 
is plotted in Figure~\ref{fig:dos}. We can see that there are
several clusters of eigenvalues for Li318 and bulk Si.
They are insulating and semiconducting, respectively.
No visible spectral gap can be observed for Graphene512. It is known 
to be a metallic system. The test cases thus encompass the full range of electronic structures from insulating to metallic.
\begin{table}[htbp]
\begin{center}
\begin{tabular}{|c|c|c|c|}
\hline 
Problem    & $n_{a}$ & ecut (Ryd) & $n_G$        \\ \hline 
 Li318     &  318  &   25         &  206,691     \\ \hline 
 Graphene  &  512  &   25         &  538,034     \\ \hline 
 bulk Si &  1000 &   35         &  1,896,173   \\ \hline 
\end{tabular}
\caption{Test problems}
\label{tab:probs}
\end{center}
\end{table}
%

\begin{figure}[htbp]
\centering
\begin{subfigure}{.33\textwidth}\includegraphics[width=\textwidth]{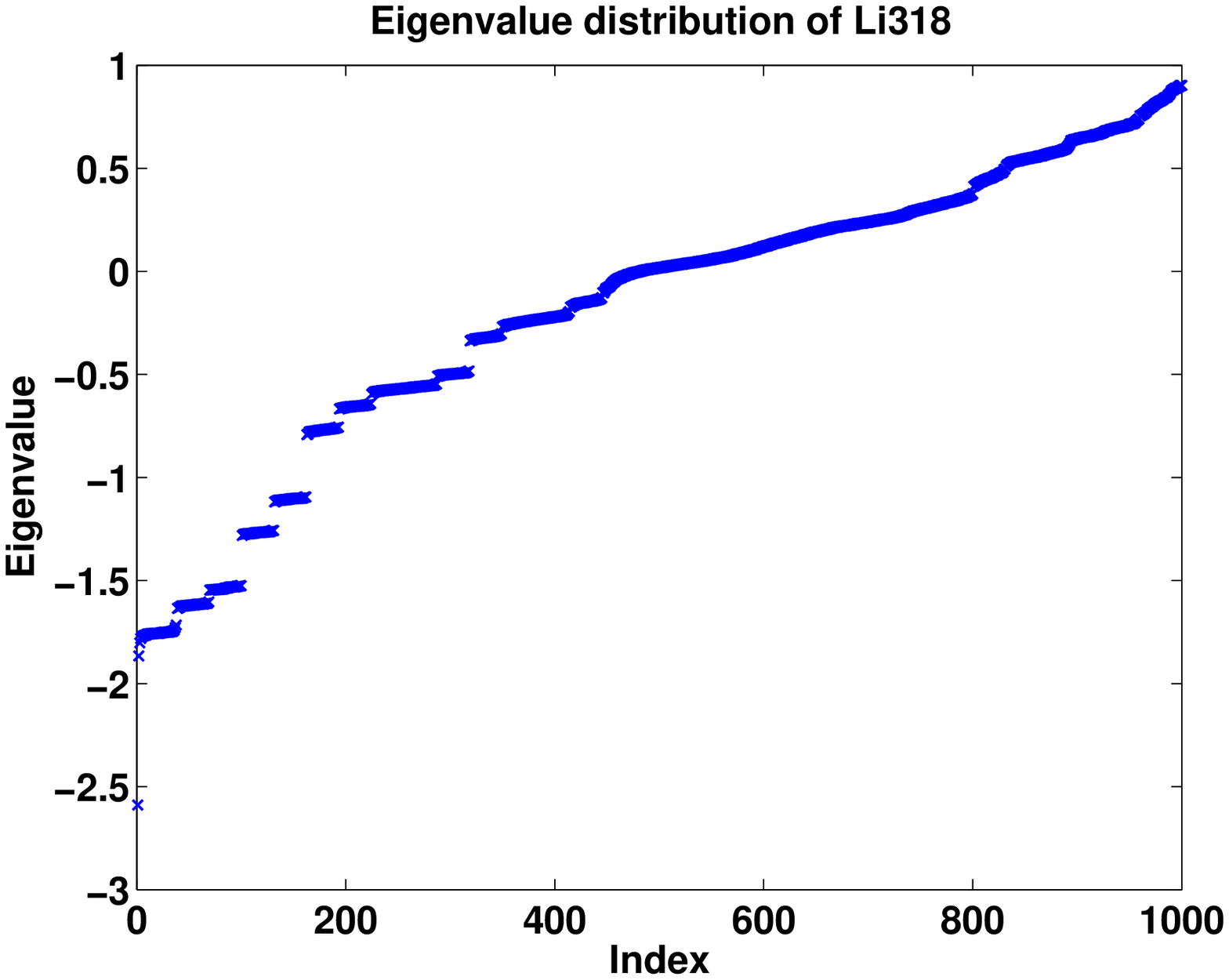}
\caption{Li318}\end{subfigure}\hfill
\begin{subfigure}{.33\textwidth}\includegraphics[width=\textwidth]{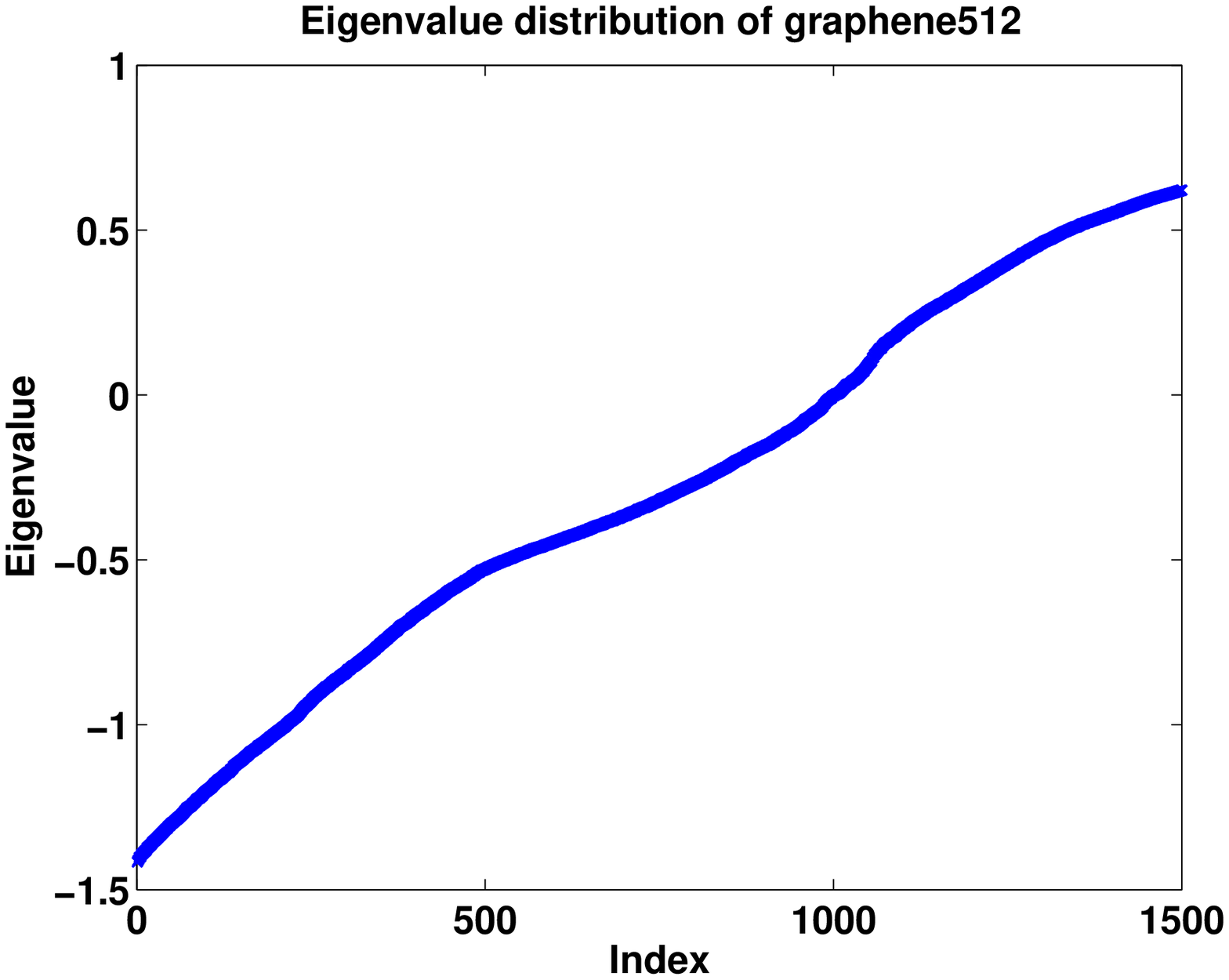}
\caption{graphene512}\end{subfigure}\hfill
\begin{subfigure}{.33\textwidth}\includegraphics[width=\textwidth]{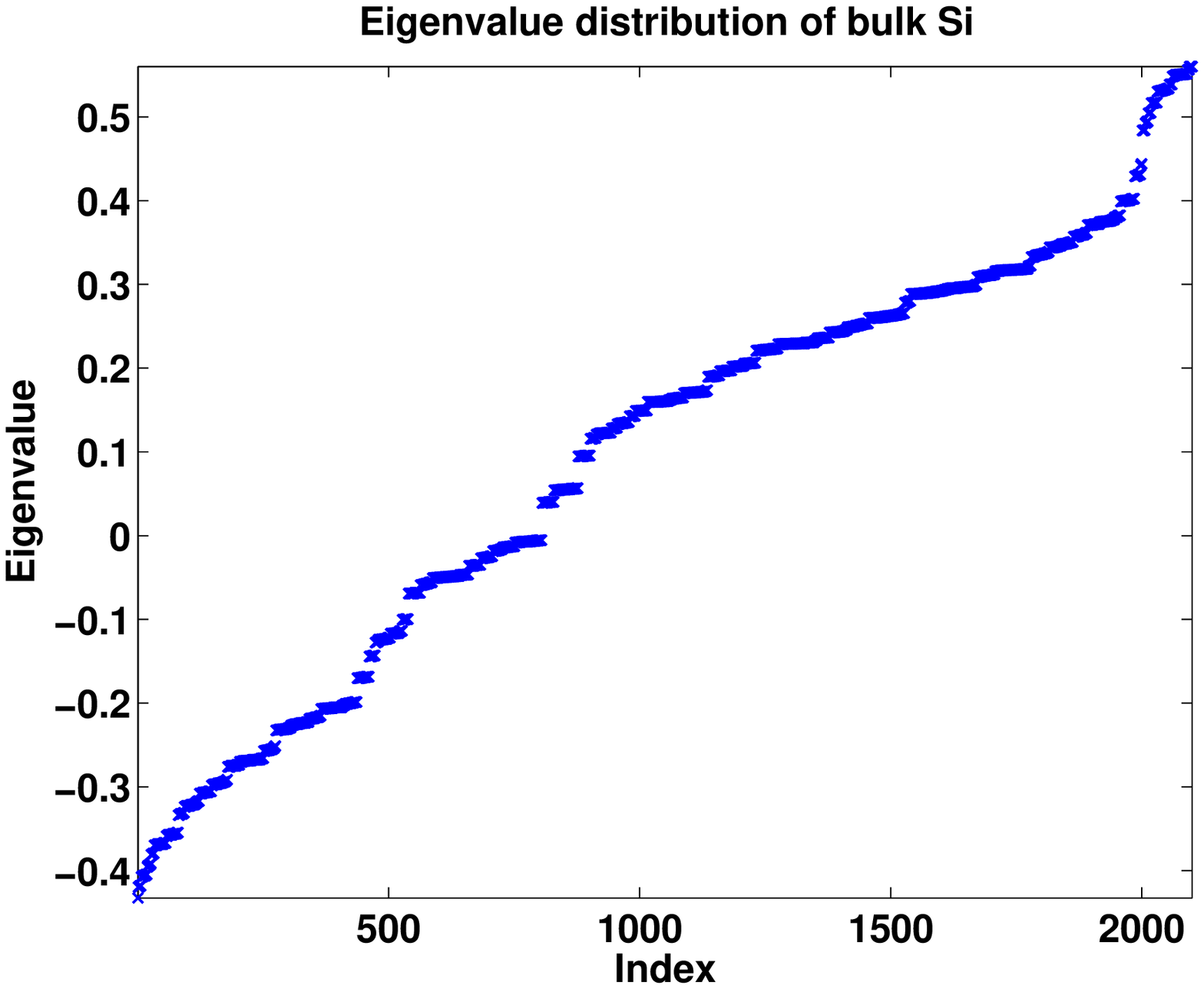}
\caption{bulk Si}\end{subfigure}
\caption{Eigenvalue distributions for test problems.}
\label{fig:dos}
\end{figure}


All tests were performed on Edison,
a Cray XC30 supercomputer maintained at the National Energy Research
Scientific Computer Center (NERSC) at Lawrence Berkeley National Laboratory.
Each node on Edison has two twelve-core Intel 2.4GHz 
``Ivy Bridge" processor sockets. It is equipped with 
64 gigabyte (GB) DDR3 1600MHz shared memory. However, 
memory access bandwidth and latency are nonuniform across all cores.  
Each core has its own 64 kilobytes 
(KB) L1 and 256 KB L2 caches. A 32MB L3 cache is shared among 
12 cores. Edison nodes are connected by a Cray Aries network with 
Dragonfly topology and 23.7 TB/s global bandwidth.

We follow the same parallelization strategy implemented in the QE 
package to perform the multiplication of the Hamiltonian and 
a wavefunction and distributed dense matrix--matrix multiplications.
The most expensive part of the Hamiltonian and wavefunction 
multiplication is the three-dimensional
FFTs. Each FFT is parallelized over $nz$ cores, where $nz$ is the
number of FFT grid points in the third dimension. Multiple FFTs 
can be carried out simultaneously on different cores when the 
``-ntg" option is used.

We do not use the multi-threaded feature of QE.
The planewave coefficients are partitioned and distributed by
rows. Therefore, the dense matrix--matrix multiplications are 
carried out by calling the DGEMM subroutine in BLAS3 on 
each processor and performing a global sum if necessarily.
The ScaLAPACK library is used to solve the dense projected 
eigenvalue problem and to perform the Cholesky factorization
required in the Cholesky QR factorization.  Because 
ScaLAPACK requires a 2D square processor grid for these
computations, a separate communication group that typically 
consists of fewer computational cores is created to 
complete this part of the computation. 
Table~\ref{tab:sqgrid} gives the default
square processor configurations generated by QE when a certain
number of cores are used to solve the Kohn-Sham problem.
Although we did not try different configurations exhaustively,
we found the default setting to be close to optimal, i.e., 
adding more cores to perform ScaLAPACK calculations generally 
does not lead to any improvement in timing because of the 
limited amount of parallelism in dense eigenvalue and
Cholesky factorization computations and the communication
overhead.
{\small
\begin{table}[htbp]
\begin{center}
\begin{tabular}{|c|c|}
\hline 
ncpus & Processor grid\tabularnewline
\hline 
\hline 
200 & 10x10\tabularnewline
\hline 
400 & 14x14\tabularnewline
\hline 
800 & 20x20\tabularnewline
\hline 
1,600 & 28x28\tabularnewline
\hline 
2,400 & 34x34\tabularnewline
\hline 
3,000 & 38x38\tabularnewline
\hline 
\end{tabular}
\caption{Default ScaLAPACK processor grid configurations used by QE 
for different total core counts.}
\label{tab:sqgrid}
\end{center}
\end{table}
}

\subsection{Band energy calculation}
We first show how PPCG performs relative to the Davidson algorithm
when they are used to solve a single linear eigenvalue problem
defined by converged electron density and its corresponding 
Kohn-Sham Hamiltonian. Table~\ref{tab:bandperf} shows
the total wall clock time required by both the block Davidson
algorithm and the PPCG algorithm for computing the $k$ lowest eigenpairs
of a converged Kohn-Sham Hamiltonian. This is often known as the
band structure calculation, although we only compute band energies 
and corresponding wavefunctions at the $\Gamma$-point of the Brillouin zone.

In all these calculations, we perform the RR procedure 
every 5 iterations. Depending on the problem, the subblock
size \textit{sbsize} is chosen to be 5 or 50. In our experience, such 
choice of \textit{sbsize} leads to satisfactory convergence behavior of the 
PPCG algorithm (we address this question in more detail below).   
For all tests, the number of buffer vectors \textit{nbuf} is set to 50.

{\small
\begin{table}
\begin{center}
\begin{tabular}{|c|c|c|c|c|c|}
\hline 
Problem & ncpus & $k$ & sbsize & Time PPCG & Time Davidson\tabularnewline
\hline 
\hline 
Li318        & 480  & 2,062 & 5 &  49 (43)  & 84 (27)  \tabularnewline
\hline 
Graphene512  & 576 & 2,254 & 50 & 97 (39)  & 144 (36) \tabularnewline
\hline 
bulk Si     & 2,400  & 2,550 & 50 & 189 (78)  & 329 (77) \tabularnewline
\hline 
\end{tabular}
\caption{Comparison of the total wall clock time (in seconds) used by PPCG and Davidson 
to compute the lowest $k$ eigenpairs. Numbers in parentheses correspond to iteration counts.}
\label{tab:bandperf}
\end{center}
\end{table}
}

For both Li318 and Graphene512, we terminate the Davidson iteration
when the relative subspace residual norm defined in~\eqref{eq:resnrm} 
is less than $tol = 10^{-2}$. Since residual norms are only calculated
when the Davidson iteration is restarted, the actual residual norm
associated with the approximate solution produced by the Davidson algorithm
may be much less than $tol$ upon termination.  We use that relative residual
norm as the stopping criterion for the PPCG algorithm. 
For bulk Si, we set $tol$ to $10^{-3}$.
 
The results shown in Table~\ref{tab:bandperf} indicate that PPCG performs
much better on the test problems than Davidson's method. We observe almost 
a factor of two speedup in terms of wall clock time.
Note that the number of iterations required by PPCG is noticeably higher 
than that of Davidson's method in some cases (e.g., Li318). However, most 
PPCG iterations are much less expensive than Davidson iterations 
because they are free of RR calculations. The reduced number of RR 
calculations leads to better overall performance. As expected, 
as the \textit{sbsize} value increases, the difference in total number of 
outer iterations between PPCG and Davidson becomes smaller.
For example, for bulk Si and Graphene512, where \textit{sbsize} is relatively large ($50$), the number of iterations taken by PPCG and Davidson are 
almost the same.  
 
\begin{figure}[htbp]
\centering
\begin{subfigure}{.33\textwidth}\includegraphics[width=\textwidth]{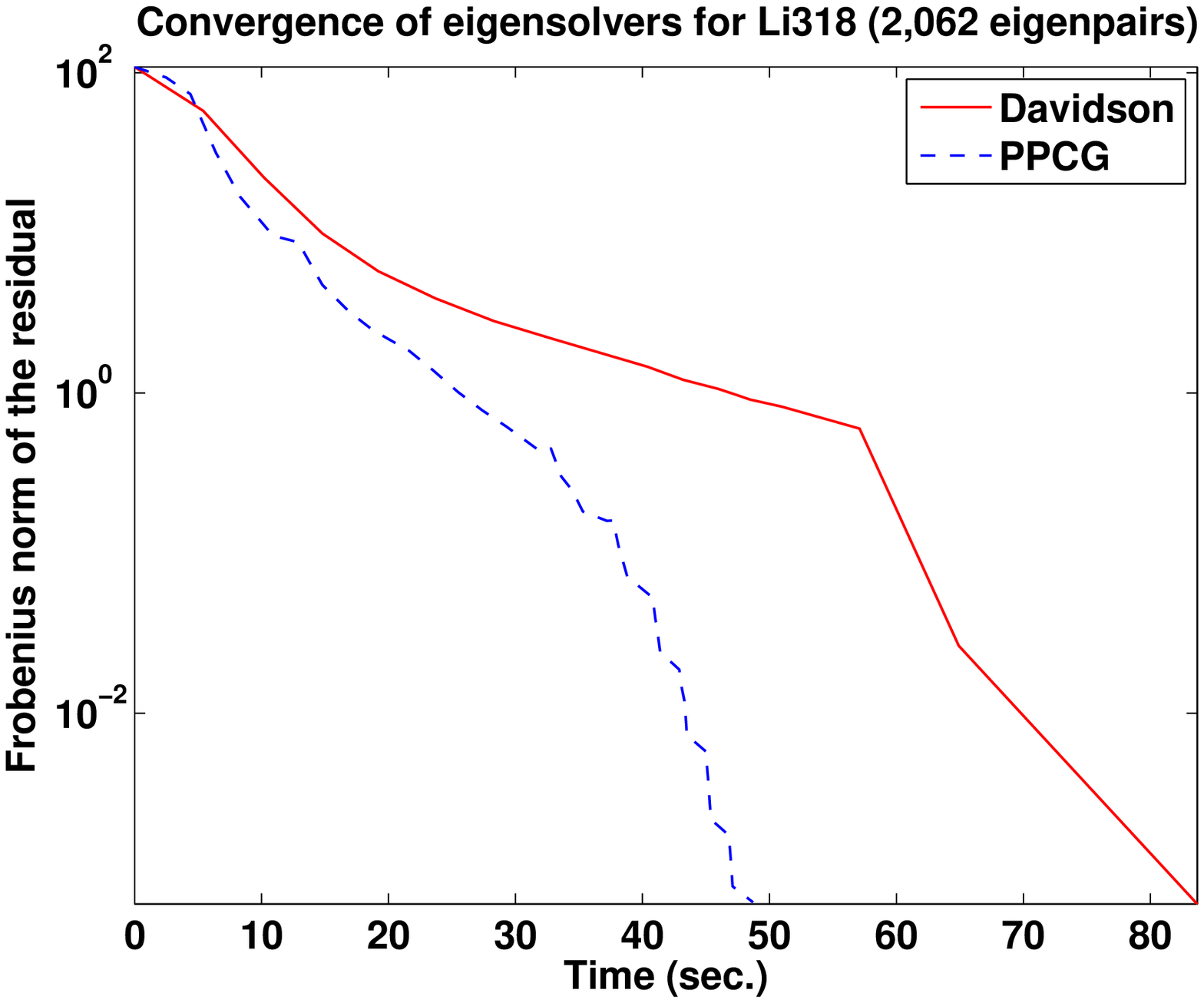}
\caption{Li318 ($k = 2062$)}\end{subfigure}\hfill
\begin{subfigure}{.33\textwidth}\includegraphics[width=\textwidth]{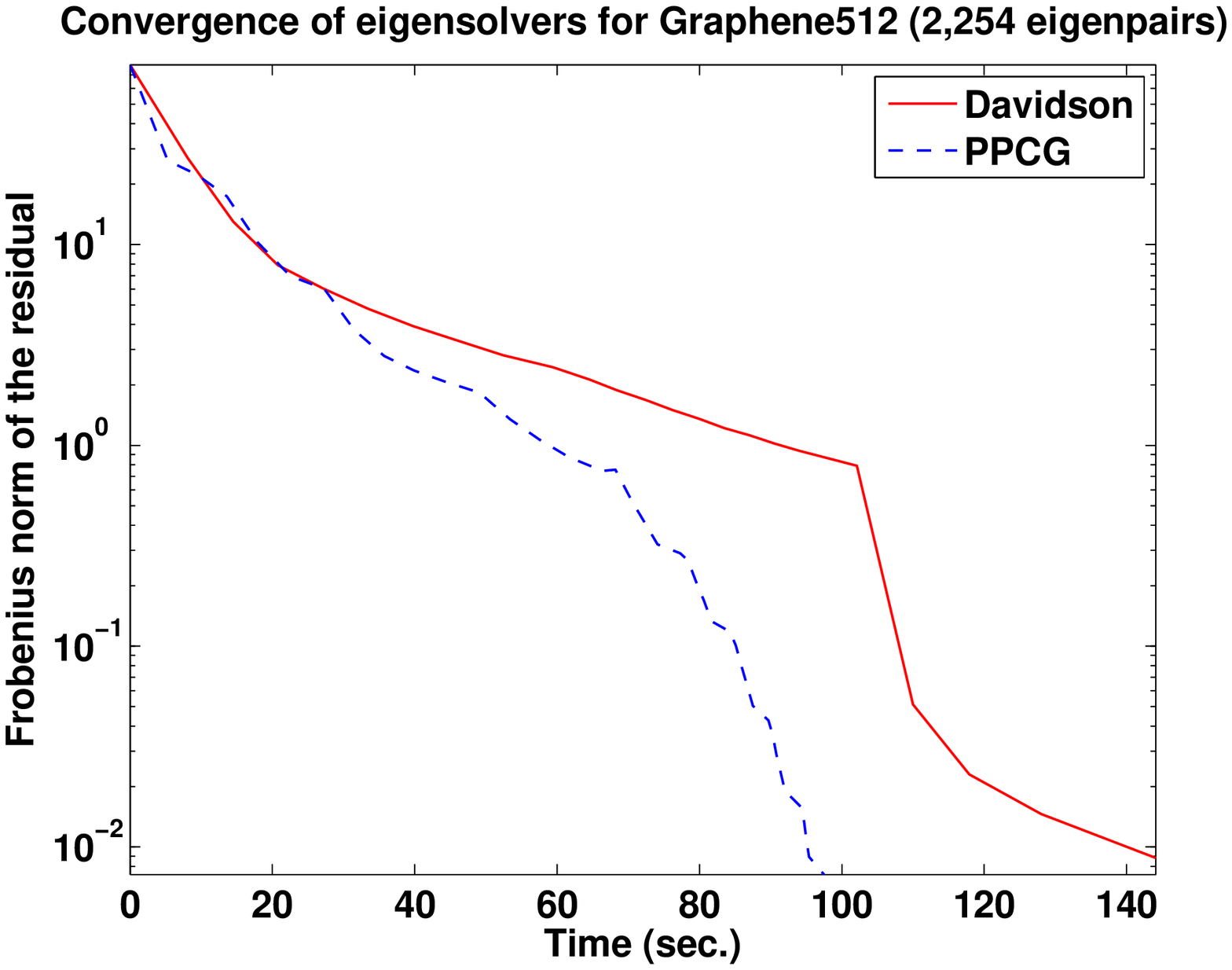}
\caption{Graphene512 ($k = 2254$)}\end{subfigure}\hfill
\begin{subfigure}{.33\textwidth}\includegraphics[width=\textwidth]{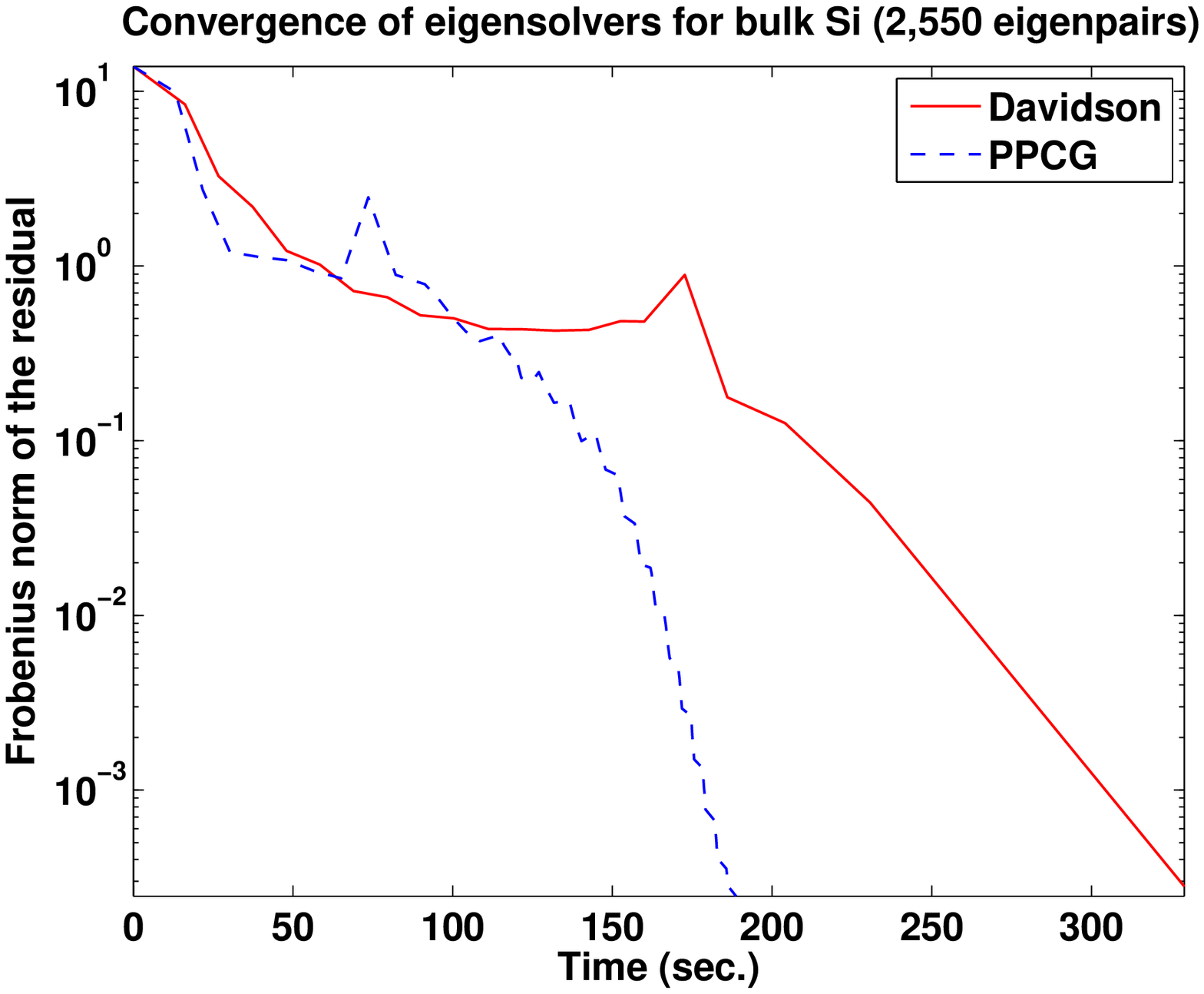}
\caption{bulk Si ($k = 2550$)}\end{subfigure}
\caption{Convergence of the Davidson and PPCG algorithms for band structure calculation.}
\label{fig:Liband}
\end{figure}

Figure~\ref{fig:Liband} shows how the Frobenius norms of the subspace
residuals change with respect to the elapsed time for
all three test problems reported in Table~\ref{tab:bandperf}.
Note that at some point, both PPCG and the Davidson method start to 
converge more rapidly. The change in convergence rate is the result of 
locking the converged eigenpairs, which significantly reduces 
computational cost in performing $AX$.  

In Tables~\ref{tab:Li318prof} and~\ref{tab:Siprof}, we provide a more 
detailed timing breakdown for both PPCG and Davidson algorithms when 
they are used to solve Li318 and bulk Si, respectively. We can clearly see that the wall clock time 
consumed by the Davidson run is dominated by RR calculations.
The RR cost is significantly lower in PPCG. However, such 
a reduction in RR cost is slightly offset by the additional cost 
of performing Cholesky QR, which we enable in each PPCG iteration. 
Its cost represents roughly 15\% of the total.
As has been discussed in Section~\ref{subsec:qr}, the number of these factorizations can, however, be further reduced, which
will lead to even more efficient PPCG.

{\small
\begin{table}[htbp]
\begin{center}
\begin{tabular}{|c|c|c|}
\hline 
Computation  & PPCG & Davidson \tabularnewline
\hline 
\hline 
 GEMM   & 16 &  11  \tabularnewline
\hline 
  $AX$  & 10  & 6 \tabularnewline
\hline 
  RR    & 13 &  66  \tabularnewline
\hline 
 CholQR & 8  &  0  \tabularnewline
\hline 
\end{tabular}
\caption{Timing profiles (in seconds) for PPCG and Davidson
when they are used to compute the 2,062 lowest eigenpairs of
the Li318 problem on 480 cores.}
\label{tab:Li318prof}
\end{center}
\end{table}
}

{\small
\begin{table}[htbp]
\begin{center}
\begin{tabular}{|c|c|c|}
\hline 
Computation  & PPCG & Davidson \tabularnewline
\hline 
\hline 
 GEMM   & 27 &  41  \tabularnewline
\hline 
  $AX$  & 94  & 96 \tabularnewline
\hline 
  RR    & 40 &  191  \tabularnewline
\hline 
 CholQR & 19  &  0  \tabularnewline
\hline 
\end{tabular}
\caption{Timing profiles (in seconds) for PPCG and Davidson
when they are used to compute the 2,550 lowest eigenpairs of
the bulk Si problem on 2,400 cores.}
\label{tab:Siprof}
\end{center}
\end{table}
}


We also note from Table~\ref{tab:Li318prof} that PPCG may spend more time in
performing dense matrix--matrix multiplications (GEMM) required to orthogonalize $W$ and $P$ against the current approximation to the desired invariant subspace 
than the Davidson algorithm.
We believe the relatively high cost of GEMM operations is due to 
the 1D decomposition of the planewave coefficient matrix used in QE, which 
is less than optimal for machines with many processors.
The performance of GEMM depends on the size of the matrices being 
multiplied on each processor.  In the case of Li318, the dimension 
of the local distributed $X$ is $720 \times 2062$, which does not lead 
to optimal single-processor GEMM performance when $X^*X$ or 
similar matrix--matrix multiplications are computed.
For bulk Si, the dimension of the local distributed 
$X$ is $2024 \times 2550$, which is nearly a perfect square.  The 
optimized BLAS on Edison is highly efficient for matrices of this
size. 

It also appears that PPCG can spend more time in performing $AX$ than 
the Davidson method.  The higher $AX$ cost in PPCG 
can be attributed to its delayed locking of converged eigenpairs.  
Because PPCG performs RR periodically, locking must also be performed
periodically even though many eigenpairs may have converged before
the next RR procedure is called.  Although we use a constant
RR frequency value \textit{rr\_period}, it is possible to choose 
it dynamically.  In the first few PPCG iterations in which
the number of converged eigenvectors is expected to be low, we should
not perform the RR procedure too frequently, in order to reduce the RR cost.
However, when a large number of eigenvectors start to converge, it may be 
beneficial to perform the RR procedure more frequently to lock the 
converged eigenvectors as soon as they appear, and hence reduce the number of 
sparse matrix multiplications.  In our tests, we observe that 
setting \textit{rr\_period} to a value between 5 and 10 typically yields 
satisfactory performance.
  
\begin{figure}[htbp]
\centering
\includegraphics[width=6cm]{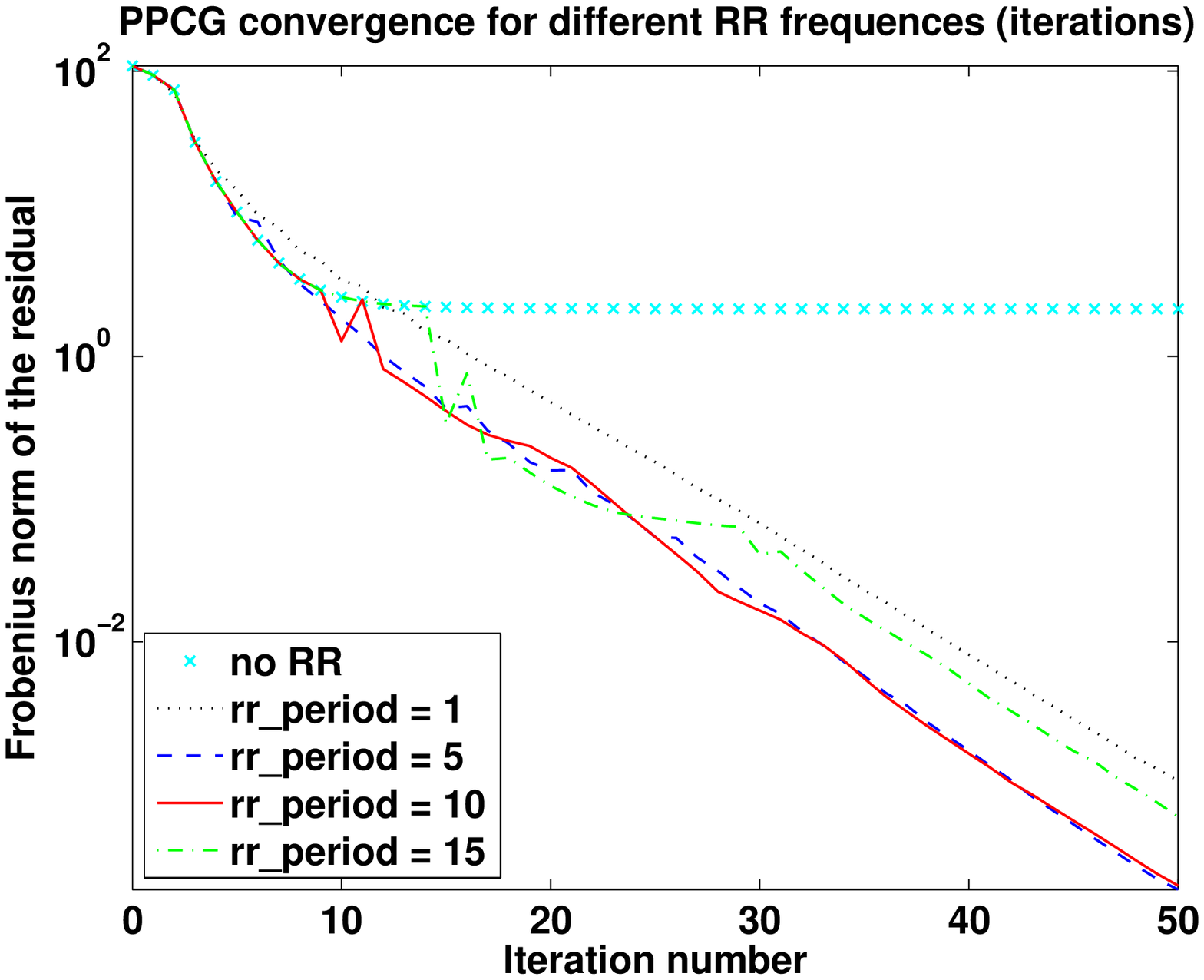}
\includegraphics[width=6cm]{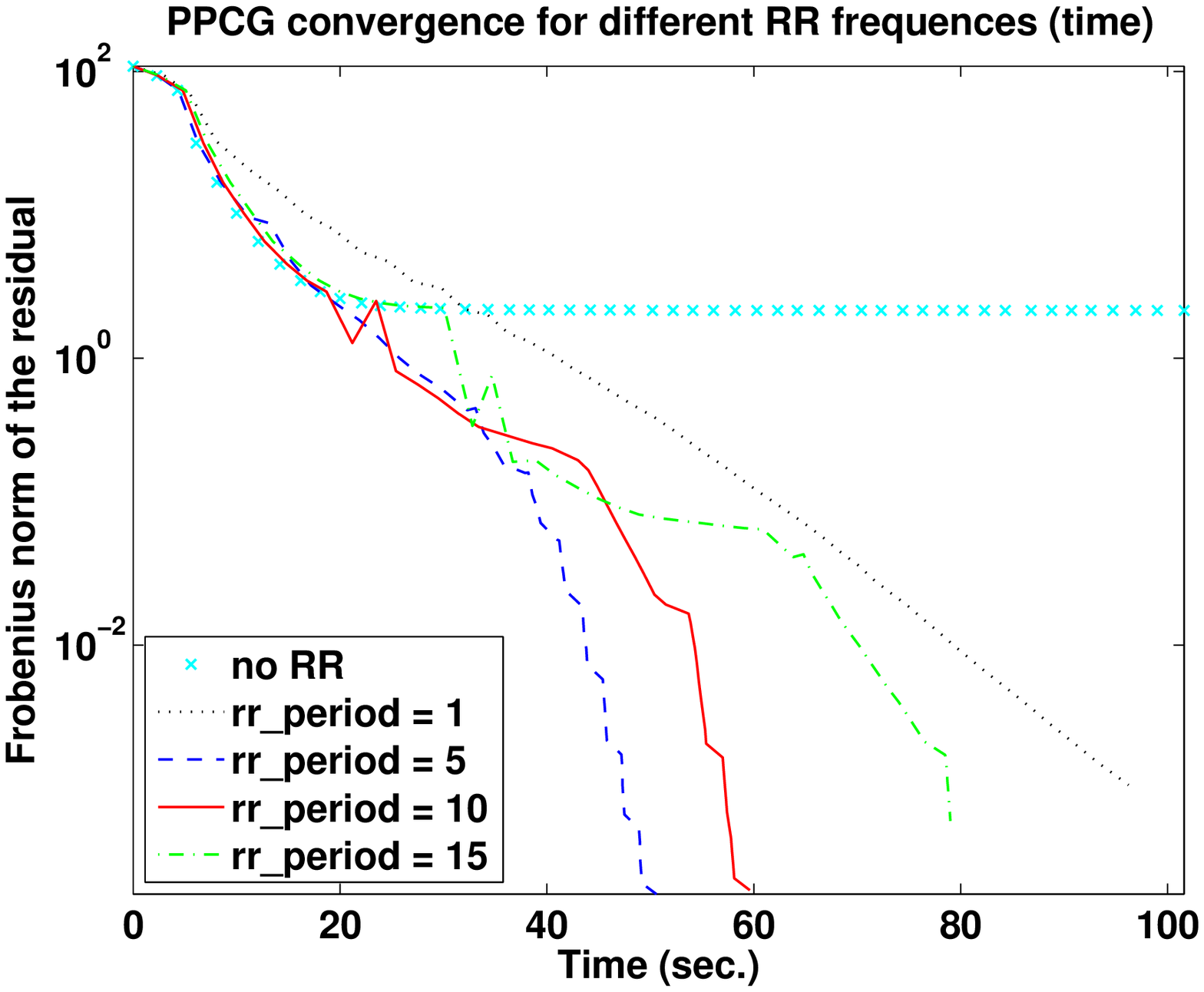}
\caption{Effect of RR frequency on the convergence of 
PPCG for the Li318 problem.}
\label{fig:rreffects}
\end{figure}
  
In Figure~\ref{fig:rreffects}, we demonstrate this finding by considering the effects of the RR frequency on the convergence of PPCG in terms of 
iteration count (left) and time (right) for the Li318 system. We can see that calling
the RR procedure periodically is crucial for reaching convergence. 
Completely removing the RR calculation generally leads to a (near) stagnation 
of the algorithm. At the same time, it can be seen from 
Figure~\ref{fig:rreffects} (left), that performing the RR procedure too 
frequently does not necessarily accelerate PPCG convergence. For this 
particular example, performing RR calculations every 15 iterations in fact results 
in essentially the same convergence rate as that observed in 
another PPCG run in which the RR computation is performed at every step.   

The effect of the RR frequency becomes more clear when we examine
the change of residual norm with respect to the wall clock time. 
As shown in Figure~\ref{fig:rreffects} (right),
the best \textit{rr\_period} value for Li318 is 5, i.e., invoking the 
RR procedure every five steps achieves a good balance between 
timely locking and reduction of RR computations. As mentioned above, in principle, one can vary the \textit{rr\_period} values during the solution process. 
As a heuristic, \textit{rr\_period} can be set to a relatively large 
number in the first several iterations, and then be gradually reduced 
to provide more opportunities for locking converged eigenvectors. 
We tried such a strategy for Li318 but did not observe significant 
improvement. Therefore, we 
leave \textit{rr\_period} at $5$ in all runs.

\begin{figure}[htbp]
\centering
\includegraphics[width=6cm]{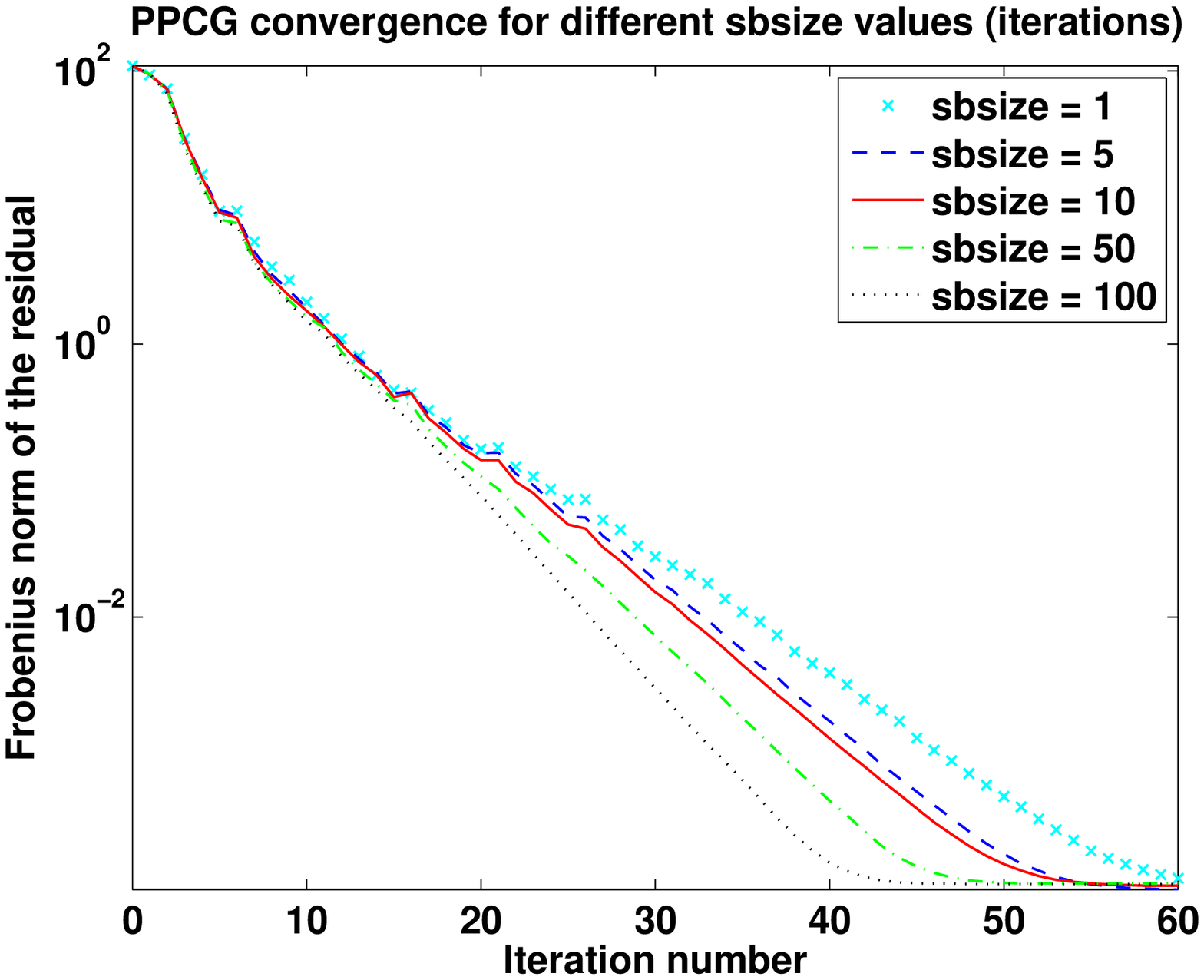} 
\includegraphics[width=6cm]{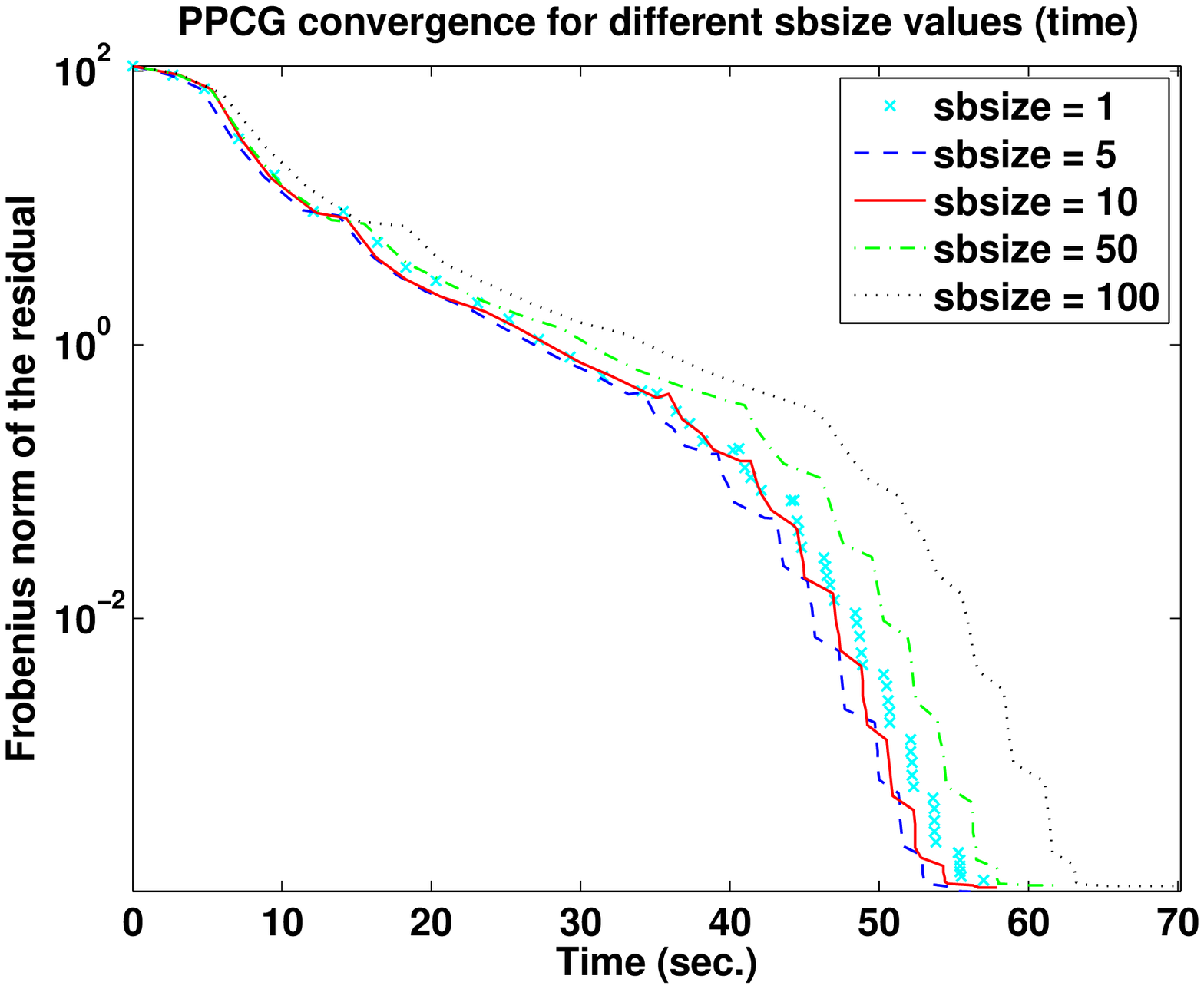}
\caption{Effect of block size on the convergence of 
PPCG for the Li318 problem.}
\label{fig:blocksize}
\end{figure}

Figure~\ref{fig:blocksize} shows the effect of the block size \textit{sbsize} on PPCG convergence for the Li318 problem.
One can see that increasing the \textit{sbsize} value results in a smaller 
number of iterations required to achieve the desired tolerance 
(Figure~\ref{fig:blocksize}, left). This behavior is expected, because in the limiting case where 
\textit{sbsize}$ = k$, PPCG becomes the LOBPCG method, which 
optimizes in the full $3k \times 3k$ subspace in each iteration.

Figure~\ref{fig:blocksize} (right) demonstrates the effect of 
\textit{sbsize} on the solution time. Remarkably, a larger 
block size, which leads to a reduced iteration count, does not 
necessarily result in better overall performance, even though
it tends to reduce outer iterations. This is in
part due to the sequential implementation of the {\tt for} loop in 
step 11 in the current implementation. Since each block minimization
in the inner loop can take a non-negligible amount of time, 
the inner {\tt for} loop can take a significant amount 
of time, even though the loop count is reduced.
On the other hand, setting \textit{sbsize}$=1$ is not desirable either, 
because of that tends to slow convergence and increase the outer PPCG iteration 
count. Furthermore, the inner minimization cannot effectively 
take advantage of BLAS3 operations in this case.  For Li318, we observe that the 
best \textit{sbsize} value is $5$.          


In Figure~\ref{fig:pscale}, we plot how the wall clock times of PPCG and  
Davidson change with respect to the number of cores when 
applied to the Li318 and bulk Si problems.
\begin{figure}[htbp]
\centering
\includegraphics[width=6cm]{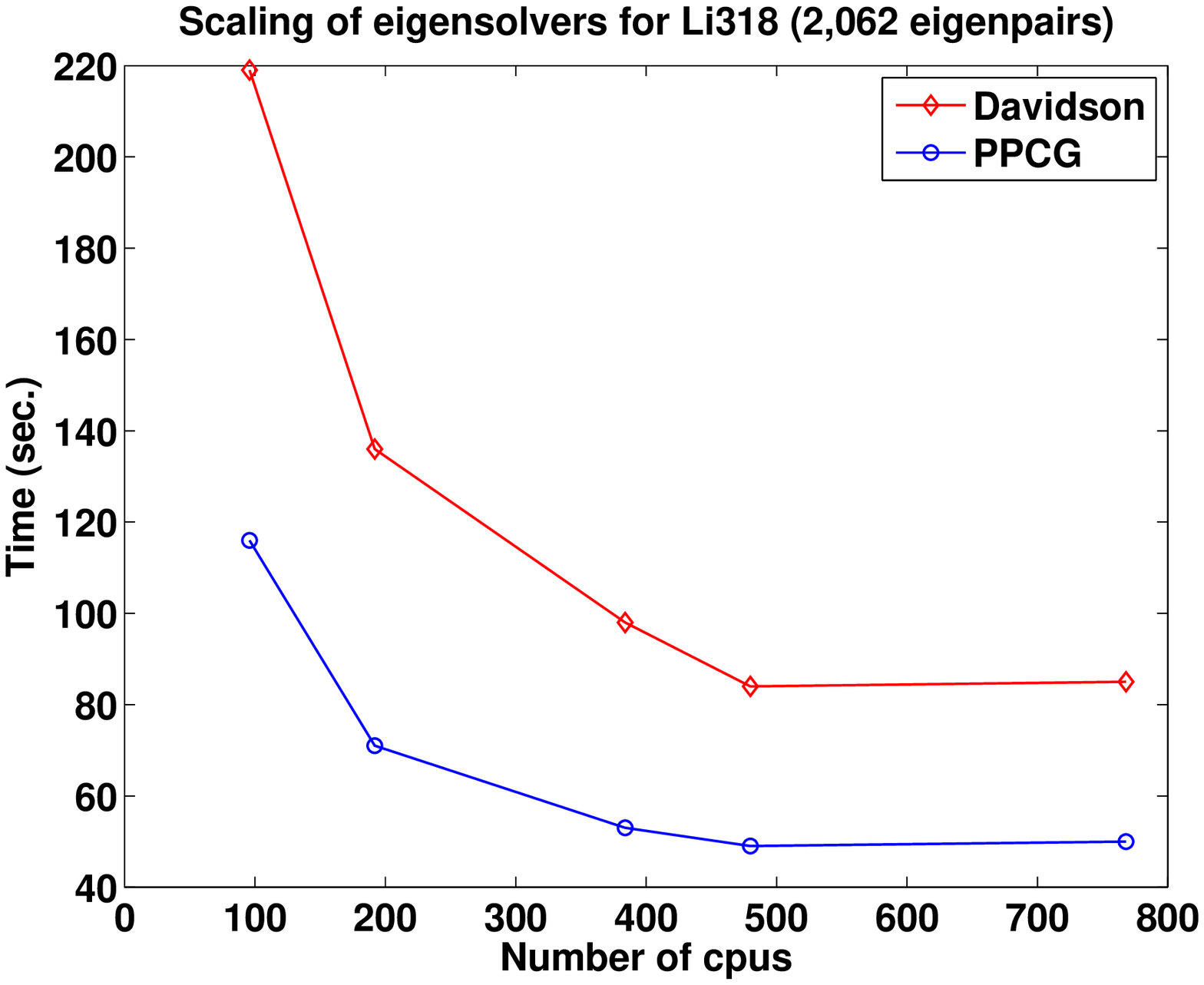}
\includegraphics[width=6cm]{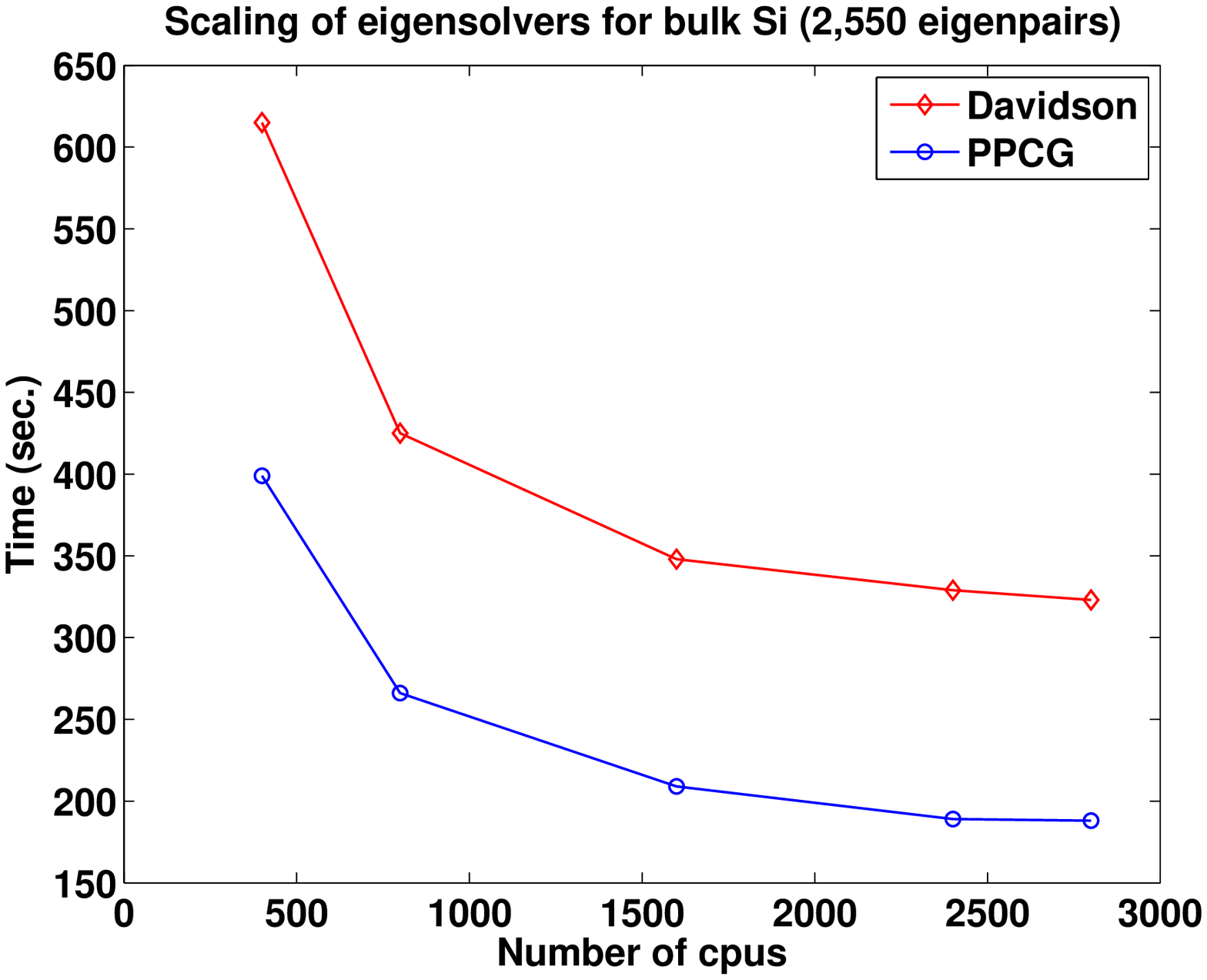}
\caption{Scaling of the Davidson and PPCG algorithms
when used to compute $2,062$ and $2,550$ bands of 
the converged Kohn-Sham Hamiltonian of the Li318 (left) and bulk Si (right) systems, respectively.}
\label{fig:pscale}
\end{figure}

We observe that both algorithms exhibit nearly
perfect parallel scalability when a relatively small number
of cores are used in the computation. However, as the 
number of cores increases, the performance of both 
algorithms stagnates. 


{\small
\begin{table}[htbp]
\begin{center}
\begin{tabular}{|c||c|c|c|c|c|c|}
\cline{2-7} 
\multicolumn{1}{c|}{} & \multicolumn{6}{c|}{ncpus} \tabularnewline
\cline{1-7} 
\multicolumn{1}{|c|}{Computation} & 48 & 96 & 192 & 384 & 480 & 768 \tabularnewline
\hline 
\hline 
GEMM & 77 & 44 & 26 & 18 & 16 & 16  \tabularnewline
\hline 
$AX$ & 33 & 28 & 17 & 13 & 10 & 11  \tabularnewline
\hline 
RR & 40 & 24 & 16 & 13 & 13 & 13  \tabularnewline
\hline 
CholQR & 28 & 16 & 9 & 8 & 8 & 9 \tabularnewline
\hline 
Total & 186 & 116 & 71 & 53 & 49 & 50 \tabularnewline
\hline 
\end{tabular}
\caption{Scaling of different computational components of PPCG for Li318.}
\label{tab:li318_ppcg}
\end{center}
\end{table}
}

{\small
\begin{table}[htbp]
\begin{center}
\begin{tabular}{|c||c|c|c|c|c|c|}
\cline{2-7} 
\multicolumn{1}{c|}{} & \multicolumn{6}{c|}{ncpus} \tabularnewline
\cline{1-7} 
\multicolumn{1}{|c|}{Computation} & 48 & 96 & 192 & 384 & 480 & 768 \tabularnewline
\hline 
\hline 
GEMM & 54 & 29 & 17 & 12 & 11 & 11\tabularnewline
\hline 
$AX$ & 22 & 19 & 13 & 8 & 6 & 7\tabularnewline
\hline 
RR & 370 & 170 & 105 & 77 & 66 & 66\tabularnewline
\hline 
Total & 449 & 219 & 136 & 98 & 84 & 85\tabularnewline
\hline 
\end{tabular}
\caption{Scaling of different computational components of the Davidson algorithm for Li318.}
\label{tab:li318_dav}
\end{center}
\end{table}
}

A closer look at the timing profiles consisting of wall clock
time used by different computational kernels, as shown in 
Tables~\ref{tab:li318_ppcg}--\ref{tab:sicluster_dav}, 
reveals that the lack of scalability at high core count is 
caused by the poor parallel scaling of both $AX$ and $GEMM$ calculations
at such core counts. A similar picture is observed for other cases 
we have tested.

{\small
\begin{table}[htbp]
\begin{center}
\begin{tabular}{|c||c|c|c|c|c|c|}
\cline{2-7} 
\multicolumn{1}{c|}{} & \multicolumn{6}{c|}{ncpus} \tabularnewline
\cline{1-7} 
\multicolumn{1}{|c|}{Computation} & 200 & 400 & 800 & 1,600 & 2,400 & 2,800 \tabularnewline
\hline 
\hline 
GEMM & 202 & 104 & 57 & 35 & 27 & 26  \tabularnewline
\hline 
$AX$ & 247 & 165 & 129 & 106 & 94 & 92  \tabularnewline
\hline 
RR & 142 & 77 & 48 & 40 & 40 & 41  \tabularnewline
\hline 
CholQR & 66 & 91 & 21 & 18 & 19 & 21 \tabularnewline
\hline 
Total & 685 & 399 & 266 & 209 & 189 & 188 \tabularnewline
\hline 
\end{tabular}
\caption{Scaling of different computational components of PPCG for bulk Si.}
\label{tab:sicluster_ppcg}
\end{center}
\end{table}
}

{\small
\begin{table}[htbp]
\begin{center}
\begin{tabular}{|c||c|c|c|c|c|c|}
\cline{2-7} 
\multicolumn{1}{c|}{} & \multicolumn{6}{c|}{ncpus} \tabularnewline
\cline{1-7} 
\multicolumn{1}{|c|}{Computation} & 200 & 400 & 800 & 1,600 & 2,400 & 2,800 \tabularnewline
\hline 
\hline 
GEMM & 248 & 138 & 76 & 47 & 41 & 38  \tabularnewline
\hline 
$AX$ & 253 & 169 & 133 & 111 & 96 & 96  \tabularnewline
\hline 
RR & 474 & 303 & 214 & 189 & 191 & 189 \tabularnewline
\hline 
Total & 986 & 615 & 425 & 348 & 329 & 323 \tabularnewline
\hline 
\end{tabular}
\caption{Scaling of different computational components of Davidson's algorithm for bulk Si.}
\label{tab:sicluster_dav}
\end{center}
\end{table}
}

The less than satisfactory scalability of GEMM is likely due 
to the 1D partition of the planewave coefficients in QE.
We are aware of a recent change in the QE design to allow 
planewave coefficients to be distributed on a 2D
processor grid such as the one used in ABINIT~\cite{ABINIT1, ABINIT2, Bottin.Leroux.Knyazev.Zerah:08} and 
Qbox~\cite{QBOX}. However, the new version of the code is 
still in the experimental stage at the time of this writing.
Hence we have not tried it.  Once the new version of QE becomes
available, we believe the benefit of using PPCG to 
compute the desired eigenvectors 
will become even more substantial.

The poor scalability of $AX$ is due to the overhead related
to the all-to-all communication required in 3D FFTs. When a
small number of cores are used, this overhead is relatively
insignificant due to the relatively large ratio of 
computational work and communication volume.  However, when a 
large number of cores are used, the amount of computation performed 
on each core is relatively low compared to the volume of communication.
We believe that one way to reduce such overhead is to perform each
FFT on fewer than $n_z$ cores, where $n_z$ is the number of FFT grid
point in the third dimension. However, this would require
a substantial modification of the QE software.


\subsection{SCF calculation}
We ran both the block Davidson (Algorithm~\ref{alg:davidson}) and the 
new PPCG algorithm to compute the solutions to the Kohn-Sham equations 
for the three systems list in Table~\ref{tab:probs}. To account 
for partial occupancy at finite temperature, we set the number 
of bands to be computed to $k = 886$ for Li318; $k = 1,229$ 
for Graphene512; and $k = 2,000$ for bulk Si.

In general, it is not necessary to solve the linear
eigenvalue problem to high accuracy in the first few 
SCF cycles because the Hamiltonian itself has not converged.  
As the electron density and Hamiltonian 
converge to the ground state solution, we should gradually 
demand higher accuracy in the solution to the linear eigenvalue problem. 
However, because the approximate invariant subspace obtained
in the previous SCF iteration can be used as a good starting
guess for the eigenvalue problem produced in the current 
SCF iteration, the number of Davidson iterations required to 
reach high accuracy does not necessarily increase.
The QE implementation of Davidson's algorithm uses
a heuristic to dynamically adjust the convergence tolerance 
of the approximate eigenvalues as the electron density and Hamiltonian
converge to the ground-state solution. In most cases, the 
average number of Davidson iterations taken in each SCF cycle is
around 2.  We have not implemented the same heuristic for
setting a dynamic convergence tolerance partly because we do not 
always have approximate eigenvalues. To be comparable to 
the Davidson solver, we simply set the maximum number of
iterations allowed in PPCG to 2. In all our test cases, two
PPCG iterations were taken in each SCF cycle.

{\small
\begin{table}
\begin{center}
\begin{tabular}{|c|c|c|c|c|}
\hline 
Problem & ncpus & sbsize & PPCG & Davidson\tabularnewline
\hline 
\hline 
Li318 & 480 & 5 & 35 (40) & 85 (49)\tabularnewline
\hline 
Graphene512 & 576 & 10 & 103 (54) & 202 (57)\tabularnewline
\hline 
bulk Si & 2,000 & 5 & 218 (14) & 322 (14)\tabularnewline
\hline 
\end{tabular}
\caption{Comparison of total wall clock time (in seconds) used by PPCG and Davidson algorithms to compute solutions of the Kohn-Sham equations. 
Numbers in parentheses correspond to SCF iteration counts.}
\label{tab:bestperf_scf}
\end{center}
\end{table}
}

In Table~\ref{tab:bestperf_scf}, we report the overall time 
used in both the Davidson and PPCG versions of the 
SCF iteration for all three test problems. The SCF 
convergence tolerance, which is used to terminate the SCF
iteration when the estimated total energy error predicted by
the so-called Harris--Foulkes energy functional~\cite{Har85, FouHay89} 
is sufficiently small,
is set to $10^{-6}$ for Li318 and bulk Si.
It is set to $10^{-4}$ for Graphene512 because the SCF iteration
converges more slowly for this problem, hence takes much longer to run.
Note that the SCF convergence tolerance is defined internally by QE, which 
we did not modify. Thus both the Davidson and PPCG versions of the SCF
iteration are subject to the same SCF convergence criterion.
We also use the same (``plain'') mixing and finite temperature smearing in 
both the Davidson and PPCG runs.
Note that, following the discussion in section~\ref{subsec:qr},
we omit the Cholesky QR step in PPCG for the reported runs. Since only
two eigensolver iterations are performed per SCF iteration, this did not affect convergence and
resulted in a speedup of the overall computation. The \textit{sbsize} parameter has been set to $5$ for the Li318
and bulk Si systems, and to 10 for Graphene512.
\begin{figure}[htbp]
\centering
\begin{subfigure}{.33\textwidth}\includegraphics[width=\textwidth]{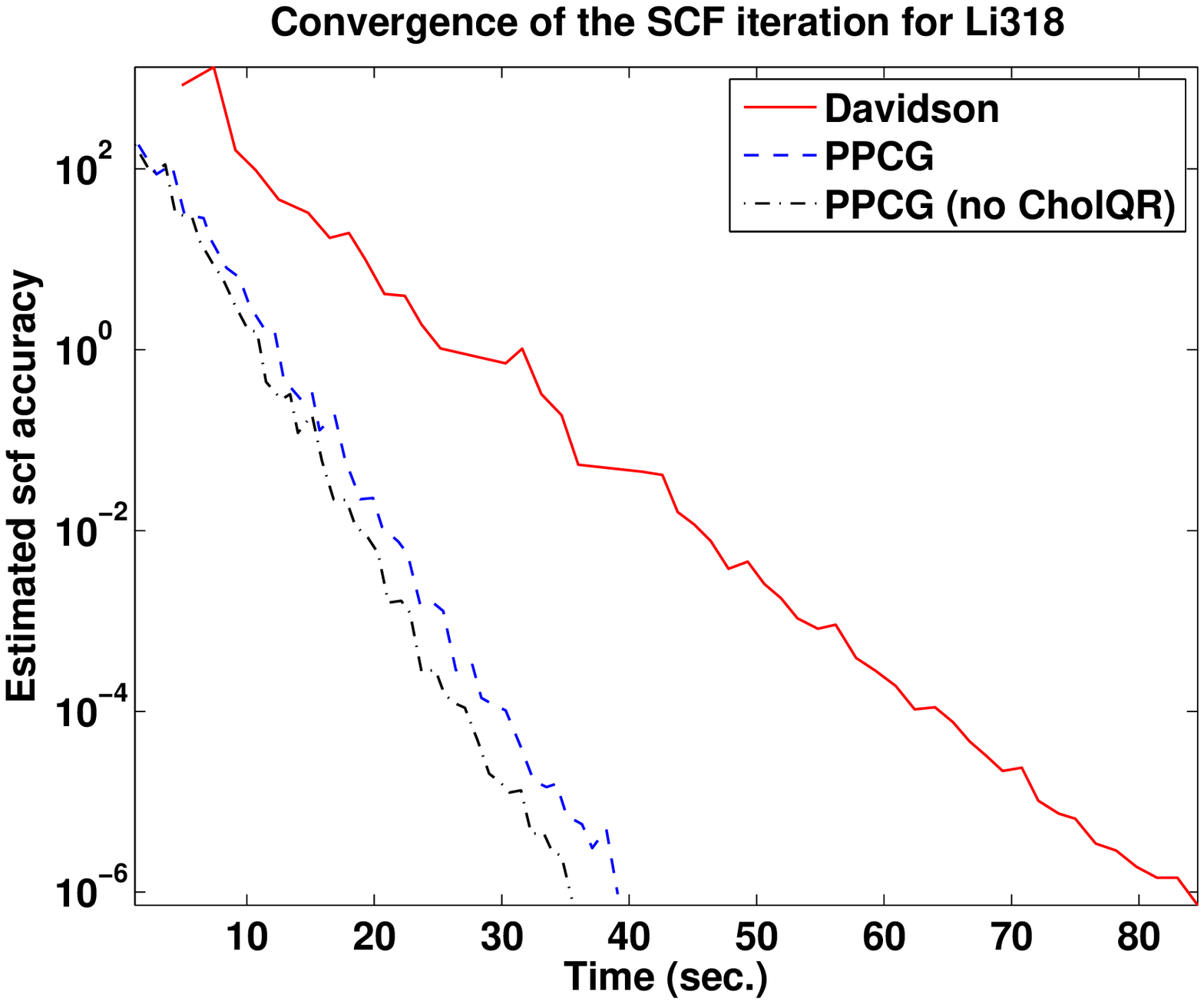}
\caption{Li318}\end{subfigure}\hfill
\begin{subfigure}{.33\textwidth}\includegraphics[width=\textwidth]{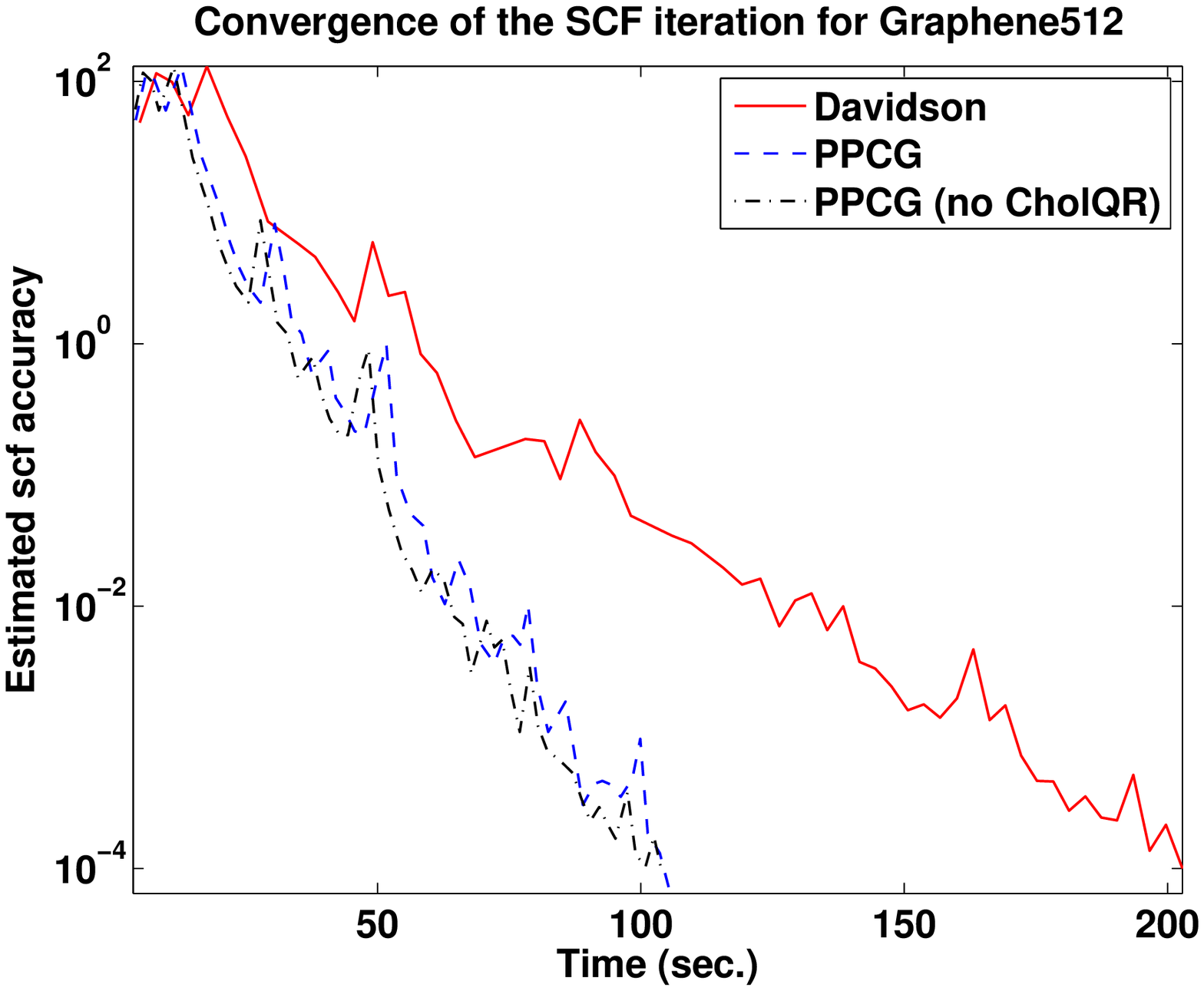}
\caption{Graphene512}\end{subfigure}\hfill
\begin{subfigure}{.33\textwidth}\includegraphics[width=\textwidth]{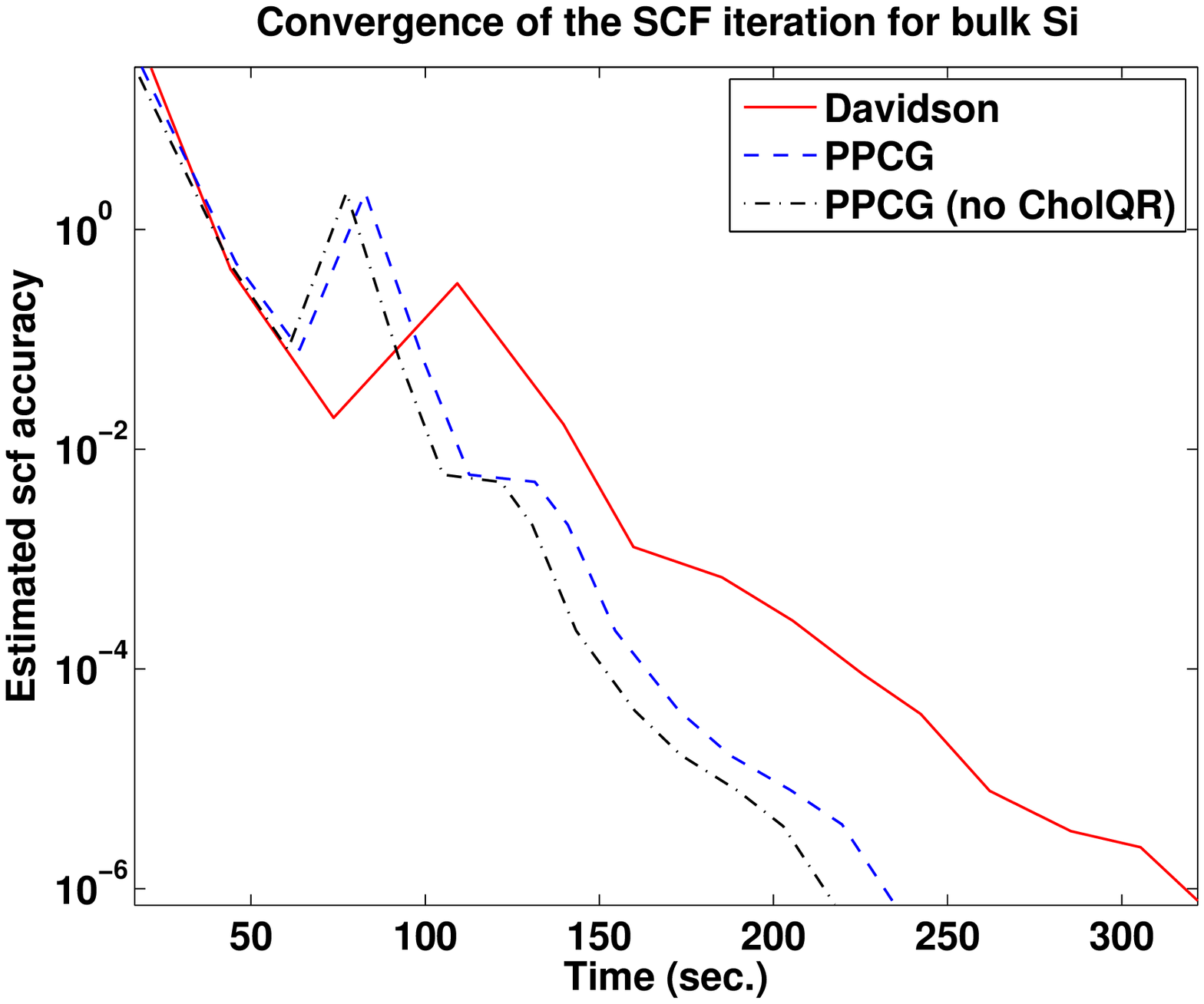}
\caption{bulk Si}\end{subfigure}
\caption{Convergence of the SCF iteration with Davidson and PPCG algorithms.}
\label{fig:scf_cv}
\end{figure}

The convergence curves corresponding to the SCF runs of Table~\ref{tab:bestperf_scf} are shown in Figure~\ref{fig:scf_cv}.   
We can clearly see that the PPCG based SCF iteration can be
nearly twice as fast as the Davidson based iteration. The figure also demonstrates the effects of skipping the Cholesky QR step, which further reduces the
PPCG run time.
It appears that, for the Li318 and Graphene512 examples, 
a slightly fewer number of SCF iterations is needed to reach convergence 
when PPCG is used to solve the linear eigenvalue problem in 
each step. 


%


In Figure~\ref{fig:pscale_scf}, we compare scalability of the SCF iteration based on the PPCG (without Cholesky QR) and Davidson algorithms.
We report the results for Li318 (left) and bulk Si (right). Similar to the case of the band structure calculations, both schemes scale
approximately up to the same number of cores for Li318, whereas scalability of PPCG is slightly better for the bulk Si example.


%
\begin{figure}[htbp]
\centering
\includegraphics[width=6cm]{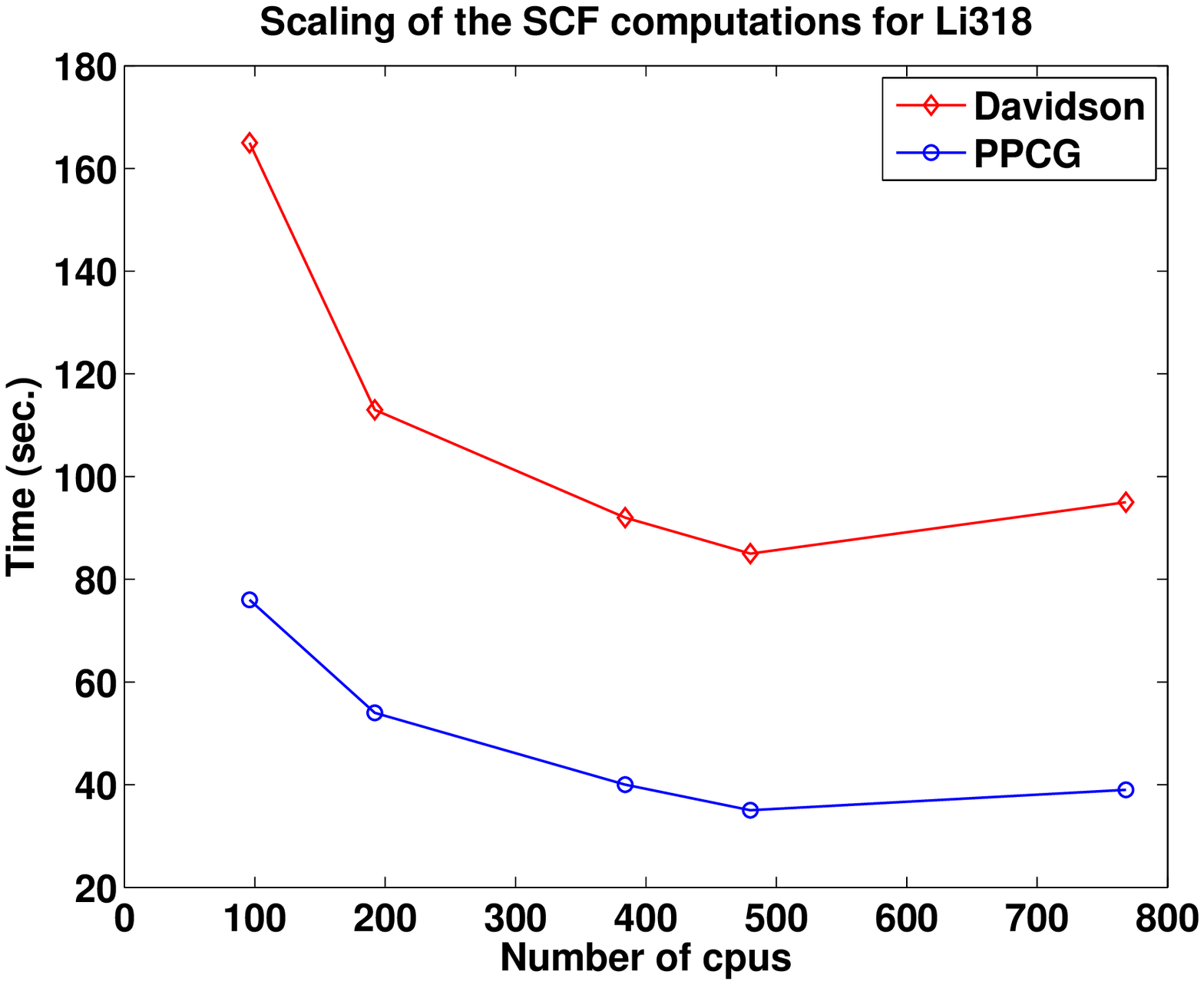}
\includegraphics[width=6cm]{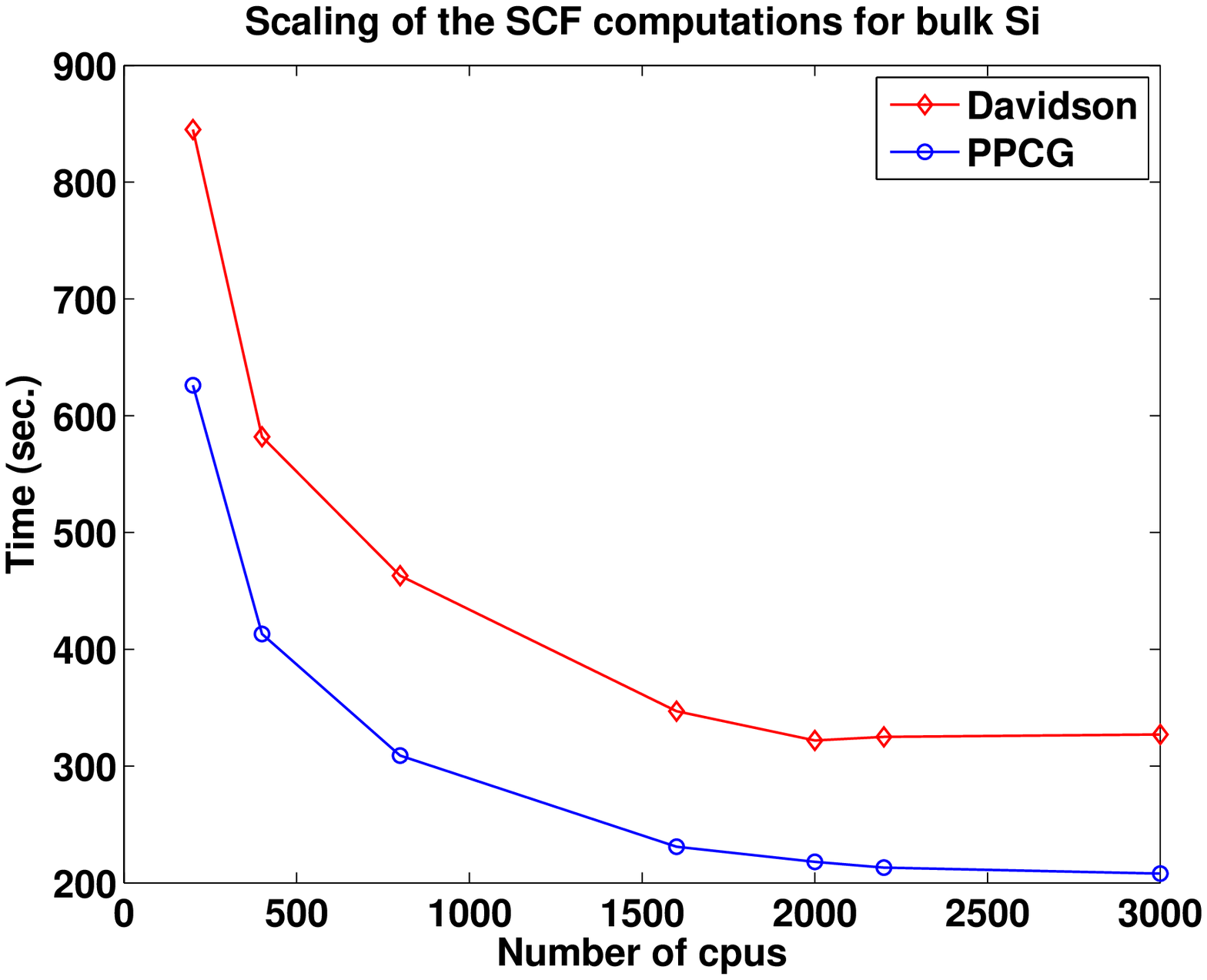}
\caption{Scaling of SCF iterations with the Davidson and PPCG algorithms
for Li318 (left) and bulk Si (right) systems.}
\label{fig:pscale_scf}
\end{figure}

\section{Conclusions}\label{sec:concl}
We presented a projected preconditioned conjugate gradient (PPCG) algorithm 
for computing an invariant subspace associated with the smallest eigenvalues 
of a large Hermitian matrix. 
The key feature of the new algorithm is that it performs
fewer Rayleigh-Ritz computations, which are often the bottleneck in 
iterative eigensolvers when the number of required eigenpairs is relatively 
large (e.g., over thousands).  
We discussed a number of practical issues that must be addressed in order to
implement the algorithm efficiently. 

We implemented the PPCG algorithm within the widely used Quantum Espresso 
(QE) planewave pseudopotential electronic structure software package. 
We demonstrated that PPCG is nearly two times faster than the existing 
state-of-the-art Davidson algorithm implemented in QE for a number of test problems.
We believe further performance gains can be achieved in PPCG relative 
to other algorithms if the multiplication
of $A$ with a block of vectors $X$ and the dense matrix multiplications
such as $X^{\ast}X$ are implemented in a scalable fashion.

\paragraph{Acknowledgments.}
The authors thank Dr.~Erik~Draeger at the Lawrence Livermore National Laboratory for insightful comments and discussions.

\appendix
\section{Detailed description of the PPCG algorithm}

In this appendix, we summarize the practical aspects related to implementation of PPCG, discussed in Section~\ref{sec:practical}, 
in Algorithm~\ref{alg:ppcg}. Note that if storage for $AX$, $AW$, and $AP$ is available, then the method can be implemented using one matrix--block multiplication 
and one block preconditioning operation per iteration. 
If only an invariant subspace is needed on 
output, then the last step of the algorithm can be omitted. 

\begin{algorithm}
\begin{small}
\begin{center}
  \begin{minipage}{5in}
\begin{tabular}{p{0.5in}p{4.5in}}
{\bf Input}:  &  \begin{minipage}[t]{4.0in}
The matrix $A$, a preconditioner $T$, a starting guess of the invariant subspace $X^{(0)} \in \IC^{n \times k}$ 
associated with the $k$ smallest eigenvalues of $A$, $X^{(0)*}X^{(0)} = I$,  parameter \textit{rr\_period} to control
the RR frequency, the splitting parameter \textit{sbsize}, and the number \textit{nbuf} of buffer vectors; 
                  \end{minipage} \\
{\bf Output}:  &  \begin{minipage}[t]{4.0in}
                 Approximate eigenvectors $X \in \IC^{n \times k}$ 
                 associated with the $k$ smallest eigenvalues $\Lambda$ of $A$;
                  \end{minipage}
\end{tabular}
\begin{algorithmic}[1]
\STATE $X \gets X^{(0)}$; $X_{\text{lock}} \gets \lt[ \ \rt]$; $P \gets \lt[ \ \rt]$; $\text{iter} \gets 1$;
\STATE Add \textit{nbuf} buffer vectors to $X$, $k \gets k + $\textit{nbuf}; $k_{\text{act}} \gets k$; 
\STATE Initialize index sets $J = \{1,\ldots, k\}$ and $J_{\text{lock}} \gets \lt[ \ \rt]$; 
\STATE Compute the initial subspace residual $W \gets AX - X(X^*AX)$; 
\STATE Use \textit{sbsize} to determine the number $s$ of subblocks\footnote{$s = k_{\text{act}}/$\textit{sbsize} if the remainder of the division is 0. Otherwise, $s = k_{\text{act}}/$\textit{sbsize}  $+ 1$.}.
\STATE Define the splitting $X = [X_1, \ \ldots, \ X_s]$ and $W = [W_1, \ \ldots, \ W_s]$.
\WHILE {convergence not reached
\footnote{Convergence should be monitored only for the $k$ leftmost vectors; 
in particular, the $nbuf$ buffer vectors should be excluded from computation of the norm of the block residual $AX - X(X^*AX)$.}
}
  \STATE $W \gets T W$;
  \STATE $W \gets (I - XX^*)W$ and $W \gets (I - X_{\text{lock}} X_{\text{lock}}^*)W$;
  \STATE $P \gets (I - XX^*)P$ and $P \gets (I - X_{\text{lock}} X_{\text{lock}}^*)P$; 
  \FOR {$j = 1, \ldots, s$} 
      \STATE $S \gets [X_j , W_j , P_j]$ ($P_j = 0$ if $P = [ \ ]$);
       \STATE Find eigenvectors $C = [C_X, \ C_W, \ C_P]^T$ ($C_P = 0$ if $P = \lt[ \ \rt]$) associated with the $k$ smallest eigenvalues $\Omega$ of~\eqref{eq:projev};   
      \STATE $P_j \gets W_j C_W + P_j C_P$;
      \STATE $X_{j} \gets X_j C_X + P_j$; 
  \ENDFOR
  \IF {$\text{mod}( \text{iter},\text{rr\_period} ) \neq 0$}
    \STATE Compute Cholesky factorization $X^* X = R^* R$;
    \STATE $X \gets XR\inv$;\footnote{Steps 18 and 19 can be periodically omitted to gain further efficiency.} 
    \STATE $W \gets AX - X(X^*AX)$; 
  \ELSE
    \STATE Set $S = [X, \ X_{\text{lock}}]$;
    \STATE Find eigenvectors $C$ associated with the $k$ smallest eigenvalues $\Omega$ of~\eqref{eq:projev};   
    \STATE $X \leftarrow S C$; $\Lambda \leftarrow \Omega$;
    \STATE $W \gets AX - X\Lambda$; 
    \STATE Use $W$ to determine column indices $J_{\text{lock}}$ of $X$ that correspond to converged eigenpairs;
           define the indices of active columns  $J_{\text{act}} \gets J \setminus J_{\text{lock}}$; 
    \STATE $X_{\text{lock}} \gets X(J_{\text{lock}})$; 
    \STATE $X \gets X(J_{\text{act}})$; $W \gets W(J_{\text{act}})$; $P \gets P(J_{\text{act}})$;
    \STATE Set $k_{\text{act}}$ to the number of active columns; 
    \STATE Compute the number $s$ of subblocks for splitting the active columns (similar to step 5); 
    \STATE Define the splitting of active columns $X = [X_1, \ \ldots, \ X_s]$, $W = [W_1, \ \ldots, \ W_s]$, and $P = [P_1, \ \ldots, \ P_s]$.
  \ENDIF
  \STATE $\text{iter} \gets \text{iter} + 1$;
\ENDWHILE
\STATE Perform steps 22-24 to obtain final eigenpair approximations ($X,\Lambda$).
\end{algorithmic}
\end{minipage}
\end{center}
\end{small}
  \caption{The PPCG algorithm (detailed description of Algorithm~\ref{alg:ppcg0})}
  \label{alg:ppcg}
\end{algorithm}

\pagebreak
\bibliography{eig}

\end{document}